\documentclass[a4paper, 12pt, reqno]{amsart}
\pdfoutput=1
\usepackage{amsmath, amssymb, amsthm, amsfonts}
\usepackage[mathscr]{eucal}
\usepackage{verbatim}
\usepackage{enumerate}
\usepackage{color}

\newcommand{\Z}{\mathbb{Z}}
\newcommand{\C}{\mathbb{C}}
\newcommand{\R}{\mathbb{R}}

\newcommand{\Hilb}{\mathcal{H}}
\renewcommand{\H}{\mathcal{H}}
\newcommand{\K}{\mathcal{K}}
\newcommand{\norm}[2]{\left\| #2 \right\|_{#1}}

\newtheorem{lemma}{Lemma}[section]
\newtheorem{theorem}[lemma]{THEOREM}
\newtheorem{definition}[lemma]{Definition}
\newtheorem{corollary}[lemma]{Corollary}

\newtheorem{remark}[lemma]{Remark}

\usepackage{enumitem}
\usepackage[bookmarks,bookmarksopen=false,
                  pdfauthor={Autor},%
                  pdftitle={Titel},%
                  colorlinks=true,%
                  linkcolor=blue,%
                  citecolor=blue,%
                  urlcolor=blue,%
                ]{hyperref}

\title[The $\alpha$-Modulation Transform]{The $\alpha$-Modulation Transform: Admissibility, Coorbit Theory and Frames of Compactly Supported Functions}
\author{M. Speckbacher, D. Bayer, S. Dahlke, and P. Balazs}

\begin{document}

\begin{abstract}
\noindent
The $\alpha$-modulation transform is a time-frequency transform generated by
square-integrable representations of the affine Weyl-Heisenberg group modulo suitable subgroups. 
In this paper we prove new
conditions that guarantee the admissibility of a given window function.
We also show that the generalized coorbit theory can be applied to this setting, assuming specific regularity of the windows.
This then yields canonical constructions of Banach frames and atomic decompositions in 
$\alpha$-modulation spaces. 
In particular, we prove the existence of compactly supported 
(in time domain) vectors that are admissible and satisfy all conditions within the coorbit machinery, which considerably go beyond known results. 

\end{abstract}

\maketitle
\vspace{2cm}
\textbf{Math Subject Classification:} 42C15, 42C40, 46E15 (primary), 46E35, 57S25 (secondary).
\textbf{Submitted to:} Monatshefte f{\"u}r Mathematik

\newpage


\section{Introduction}

The main purpose of  time-frequency analysis is to decompose 
given signals and functions such that their significant characteristics are revealed. 
In audio applications for 
example one seeks to unveil the time evolution of the frequency components of a piece of music. 
In many acoustical signals both tonal as well as impact sound components appear. 
 For example,  the
solution of the acoustic scattering problem can contain harmonic as well as non-smooth components, depending on
the geometry of the scatterer.
Harmonic components are well represented by Gabor systems \cite{feistro1}, transient components by wavelet systems 
\cite{daubech1}. 
So it seems natural to use a representation, which is in some sense intermediate between the Gabor and wavelet 
setting, combining the strengths of both worlds. 

In a first approach to construct such a transform, one can consider representations of the affine Weyl-Heisenberg group  
$G_{aWH}$ which contains both the 
affine group (wavelets) and the Weyl-Heisenberg group (Gabor). However, Torr{\'e}sani \cite{torr1,torr2} showed 
that the 
representations of $G_{aWH}$ are not square integrable. Thus, quotients  of $G_{aWH}$ have to be considered.
The $\alpha$-modulation transform is a particular choice of this construction. 
It depends on the parameter $\alpha\in[0,1)$, where $\alpha=0$ corresponds to 
 the Gabor setting and the limiting case $\alpha\rightarrow 1$ corresponds to a wavelet-like transform, 
 see \cite{DahForRau+2008,spexxl14}.

For specific problems in signal analysis, a discretized version of this transform, $\alpha$-modulation frames have already been  applied quite successfully
\cite{DahTes08}.
We are convinced that in the long run $\alpha$-modulation frames will  also provide excellent tools for numerical 
purposes 
such as, e.g.,  scattering problems or the numerical solution of partial differential
equations and  integral equations  on domains and manifolds. 
Indeed, very often the solutions to these equations  
contain periodic components as well as singularities on lower dimensional manifolds, so that neither Gabor frames
nor wavelet
bases can give rise to sparse representations.
 However, a sparse representation using $\alpha$-modulation frames might be possible. 
This would pave the way 
to very efficient
(adaptive) numerical algorithms in the spirit of  \cite{CohDahDev2001,CDD2,Ste}.  

Frames, i.e., generalizations of orthonormal bases \cite{Casaz1} that allow for redundant representation of
functions, can also be used for the discretization
of operators \cite{xxlframoper1}, for example  in a 
boundary element approach \cite{Gauletal03}. 
Frame methods have been used successfully in this context \cite{harbr08}, in particular in an
adaptive approach \cite{DahForRaa+2007,Ste}. 
Choosing good frames can lead to a better compressibility of the involved matrices. Also here we think that 
$\alpha$-modulation frames are a promising option.

To make the long-term goal of using $\alpha$-modulation frames in a 
boundary element discretization reachable,  it is absolutely necessary
to construct them to be {\em compactly supported}. Otherwise it would be highly complicated or 
maybe even impossible to treat bounded domains and the efficiency of numerical solution might be hard to judge. 
This is the main intention of the paper at hand.
In particular, we will make use of  generalized coorbit theory and improve the results of 
 \cite{DahForRau+2008} where it is shown that band-limited windows can be used to construct 
 $\alpha$-modulation frames.



 Feichtinger and Gr{\"o}chenig introduced coorbit theory in the late 1980's 
in a series of papers \cite{FeiGro_88,FeiGro_89,FeiGro_89-2}.
Their construction works as follows: starting from a (square) integrable group representation 
one can introduce  the (generalized) voice transform
and define the \textit{coorbit space} to be the space of 
distributions whose voice transform is contained in some solid Banach space. 
A remarkable asset of coorbit theory is that a  suitable discretization of the underlying 
group  yields  Banach frames on all coorbit spaces all at once. In order to be able to also work with groups
that are not square integrable, generalized coorbit theory was introduced in 
\cite{DahForRau+2008,dastte04,dastte04-1}. In particular, the discretization machinery still works for this 
setting.

Although the analysis presented in this paper is sometimes quite technical,  we finally end
up  with 
very natural and simple conditions on the decay of the Fourier transform.   These conditions allows for a plethora 
of 
compactly supported admissible functions. In particular, cardinal B-splines fit into this context.  It would even
be possible to use   
B-spline wavelets that possess vanishing moments and  therefore  can give rise to efficient compression strategies,
a very important step  towards efficient numerical schemes \cite{CohDahDev2001,CDD2,Ste}.

The paper is organized as follows. In Section \ref{sec:prel0} we review basic facts about generalized representation theory
modulo subgroups. In Section \ref{sec:admis0}, we state the main result on the admissibility  for the $\alpha$-modulation
transform of functions with certain decay of their Fourier transform. 
After briefly recalling the basics of generalized coorbit theory in Section 
\ref{sec:coorbit_basics} we show in Section \ref{sec:coorbit_alpha} that coorbit theory is applicable to the 
$\alpha$-modulation transform using again windows with particular regularity in Fourier domain.


\section{Preliminaries} \label{sec:prel0}


\subsection{Representation theory modulo subgroups}

We first give a short outline of the theory of square-integrable group representations introduced in  \cite{grmopa86}, see also
\cite{fo95} or \cite{wo02}.
Let $G$ be a locally compact Hausdorff topological group. It is well known that for such groups there always exists a nonzero Radon measure 
$\mu$, unique up to a constant factor, that is invariant under left translation. This measure is the so-called \emph{(left) Haar measure} 
of $G$. If the left Haar measure is simultaneously a right Haar measure as well (i.e. it is invariant under right translations), we call 
the group \emph{unimodular}. Let $\Hilb$ be a separable complex Hilbert space with inner product $\langle\cdot, \cdot \rangle$ and norm 
$\| \cdot \|$. Denote by $\mathcal{U}(\Hilb)$ the group of unitary operators on $\Hilb$. A \emph{unitary representation} of $G$ on $\Hilb$ 
is a strongly continuous group homomorphism $\pi:G \to \mathcal{U}(\Hilb)$, i.e. a mapping $\pi:G \to \mathcal{U}(\Hilb)$ such that
\begin{enumerate}[label={(\roman*)}]
\item $\pi(gh) = \pi(g)\pi(h)$ for all $g,h \in G$, and
\item for every $\psi \in \Hilb$, the mapping $G \to \Hilb$, $g \mapsto \pi(g)\psi$ is continuous.
\end{enumerate}
The group representation $\pi$ is said to be \emph{irreducible} if the only invariant subspaces under $\pi$ are $\{0\}$ and $\Hilb$, i.e. the only closed 
subspaces $M \subseteq \Hilb$ such that $\pi(g)(M) \subseteq M$ for every $g\in G$ are $M=\{0\}$ and $M = \Hilb$.
The group representation is said to be \emph{square-integrable} if it is irreducible and there exists a vector $\psi \not= 0$ in $\Hilb$
such that 
\[
c_\psi := \int_G \left| \langle \psi, \pi(g)\psi \rangle \right|^2\,d\mu(g) < \infty.
\]
Such a vector $\psi$ with $\| \psi \| = 1$ is called an \emph{admissible wavelet}. Its associated \emph{wavelet constant} is $c_\psi$. 
For an admissible wavelet $\psi$ and $f \in \Hilb$, the \emph{voice transform} or \emph{generalized wavelet transform} $V_{\psi}f:G \to 
\C$ is defined as
\[
V_{\psi} f(g) := \langle f, \pi(g)\psi \rangle,
\]
for $g \in G$. By square-integrability, we have $V_{\psi}f \in L^2(G, \mu)$ and $V_{\psi}: \Hilb \to L^2(G,\mu)$ bounded. The adjoint of
this mapping is
\[
V^\ast_\psi: L^2(G,\mu) \to \H, \ \ \  V^\ast_\psi F = \int_G F(g) \pi(g)\psi\,d\mu(g),
\]
to be understood in weak sense as
\[
\langle V^\ast_\psi F, h \rangle = \int_G F(g) \langle \pi(g)\psi, h \rangle \,d\mu(g)
\]
for all $h \in \Hilb$.
If $\pi:G \to \mathcal{U}(\Hilb)$ is  square-integrable, and $\psi\in\Hilb$ is admissible, then we have the following 
\emph{resolution of the identity} which holds weakly: for all $f \in \Hilb$,
\[
\frac{1}{c_\psi} V_\psi^\ast ( V_{\psi}f ) = \frac{1}{c_\psi} \int_G \langle f, \pi(g)\psi \rangle \pi(g)\psi\,d\mu(g) = f,
\]
that means the \emph{reproducing formula}
\[
\frac{1}{c_\psi} \int_G \langle f, \pi(g)\psi \rangle \langle \pi(g)\psi, h \rangle\,d\mu(g) = \langle f, h \rangle
\]
holds for all $f,h \in \Hilb$. One may interpret the family $\{ \pi(g)\psi\,:\,g \in G\}$ as a \emph{continuous frame}, with $V_\psi$ the
\emph{analysis operator}, $V^\ast_\psi$ the \emph{synthesis operator} and $A_{\psi} := V^{\ast}_\psi( V_\psi)$ the \emph{frame operator}.
For further reading on continuous frames, see \cite{alanga93,ranade06}.

From here, one proceeds to build classical coorbit theory, as explained in \cite{FeiGro_89, FeiGro_89-2}.

In many cases, however, representations of a group are not square-integrable. The usual informal  interpretation of this fact is that the 
group is, in a certain sense, too large. Following
\cite{DahForRau+2008}, the subsequent technique may be used to make the group smaller: choose a suitable closed subgroup $H$ and factor
out, forming the quotient $G/H$. In general, $H$ need not be a normal subgroup, so that $G/H$ will, in general, not carry a group 
structure; 
it is a \emph{homogeneous space}, though, i.e. the group $G$ acts on $G/H$ continuously and transitively by left translation. 
The quotient can always be equipped, in a natural way, with a measure $\mu$ that is \emph{quasi-invariant} under left translations, 
i.e. $\mu$ and all its left-translates have the same null sets. In many examples the measure $\mu$ will be translation-invariant in the 
first place. In order to transfer the representation from the group to the quotient, one then introduces a measurable \emph{section} 
$\sigma: G/H \to G$ which assigns a group element to each coset. We can then generalize admissibility and square-integrability for 
representations modulo subgroups in the following definition. 

\begin{definition}\label{D:admissible_abstract}
Let $G$ be a locally compact group, $\pi:G \to \mathcal{U}(\mathcal{H})$ a unitary representation, $H$ a closed subgroup of $G$, 
and $X = G/H$, equipped with a (quasi-)invariant measure $\mu$. Let $\sigma:X \to G$ be a section and $\psi \in \mathcal{H} \setminus 
\{0\}$. Define the operator $A_{\sigma}$ on $\mathcal{H}$ (weakly) by
\begin{equation}
A_{\sigma}f := \int_X \langle f, \pi(\sigma(x))\psi \rangle \pi(\sigma(x))\psi\,d\mu(x), \quad f \in \mathcal{H}.
\end{equation}
If $A_{\sigma}$ is bounded and boundedly invertible, then $\psi$ is called \textbf{admissible}, and the unitary representation 
$\pi$ is called \textbf{square-integrable modulo $(H, \sigma)$}.
\end{definition}

There is also a generalization of the voice transform in this setting.

\begin{definition}\label{D:voice_abstract}
Let $\psi$ be admissible. Then the \textbf{voice transform} of $f \in \mathcal{H}$ is defined by
\begin{equation}
V_{\psi}\,f(x) := \langle f, \pi(\sigma(x))\psi \rangle, \quad x \in X.
\end{equation}
We further define a {second transform}
\begin{equation}
W_{\psi}\,f(x) := V_{\psi}(A_{\sigma}^{-1}\,f)(x) = \langle f, A_{\sigma}^{-1}\,\pi(\sigma(x))\psi \rangle, \quad x \in X.
\end{equation}
\end{definition}

We have then the following version of the reproducing formula (see formula (2.4) in \cite{DahForRau+2008}).

\begin{theorem}
Let $\psi$ be admissible for the representation $\pi$ modulo $(H,\sigma)$. Then, for all $f_1, f_2 \in \Hilb$,
\[
\langle f_1, f_2 \rangle = \langle W_\psi f_1, V_\psi f_2 \rangle_{L^2(G,\mu)} = \langle V_\psi f_1, W_\psi f_2 \rangle_{L^2(G,\mu)}.
\]
\end{theorem}


\subsection{The Setting for the \texorpdfstring{$\alpha$}{[alpha]}-transform}

%

It turns out that the $\alpha$-modulation transform can be constructed  in the setting of generalized representation theory 
modulo subgroups for a particular group, the \emph{affine Weyl-Heisenberg group}.

Denote by $\R_+$ the set of positive real numbers. Throughout this paper we use the following convention for the Fourier transform
\[
 \hat f(\xi):=\mathcal{F}(f)(\xi):=\int_\R f(x) e^{-2\pi i\xi x}dx
\]

The \emph{affine Weyl-Heisenberg group} is the set
\[ G_{aWH} := \R^2 \times \R_+ \times \R
\]
together with composition law
\begin{align*}
(x,\omega, a, \tau) & \circ (x^{\prime}, \omega^{\prime}, a^{\prime}, \tau^{\prime}) = ( x + ax^{\prime}, \omega + \frac{1}{a}
\omega^{\prime}, aa^{\prime}, \tau + \tau^{\prime} + \omega a x^{\prime} ).
\end{align*}
Equipped with the usual product topology of the respective Euclidean topologies on $\R$ and $\R_+$, this becomes a (non-abelian)
locally compact Hausdorff topological group. The neutral element is $(0,0,1,0)$, and the inverse to $(x,\omega, a, \tau)$ is 
$(x, \omega, a, \tau)^{-1} = (-\frac{x}{a}, -\omega a, \frac{1}{a}, - \tau + x\omega)$. The Haar measure is given by
\[ d\mu(x,\omega, a,\tau) = dxd\omega\frac{da}{a}d\tau.
\]
This is in fact both a left and right Haar measure on $G_{aWH}$, thus the group is \emph{unimodular}.

We define the "basic three" operators of time-frequency analysis, \emph{translation}: $T_xf(t) = f(t-x)$, \emph{modulation}: 
$M_{\omega}f(t) = e^{2\pi i \omega t}f(t)$, and \emph{dilation}: $D_af(t) = \frac{1}{\sqrt{a}}f(\frac{t}{a})$ (with $x, \omega \in \R$ 
and $a \in \R_+$). These are unitary operators on $L^2(\R)$. Using them, we define the \emph{Stone-von Neumann representation}, given by
\[
\pi:G_{aWH} \to \mathcal{U}(L^2(\R)), \quad\quad \pi(x, \omega, a, \tau) = e^{2\pi i \tau} T_x M_{\omega} D_a.
\]
This constitutes a unitary representation of $G_{aWH}$, but unfortunately not a square-integrable one, see \cite{torr1}.

The subset
\[
H = \{ (0,0,a,\tau) \} \subseteq G_{aWH}.
\]
is a closed subgroup of the affine Weyl-Heisenberg group, although not a normal subgroup. Define the quotient
\[
X := G_{aWH} / H \simeq \R^2.
\]
This is not a group but a \emph{homogeneous space}. It carries the measure $dxd\omega$ which is in fact a truly \emph{invariant} 
measure under left translations on $X$.

For $0 \leq \alpha < 1$, choose the \emph{section (or lifting)}
\[
\sigma:X \to G_{aWH}, \quad \quad \sigma(x,\omega) = (x,\omega, \beta(\omega), 0)
\]
with
\[
\beta(\omega) = (1 + |\omega |)^{-\alpha}.
\]

One can then show \cite{DahForRau+2008,spexxl14} that, for $\psi \in L^2(\R)$, the operator from Definition \ref{D:admissible_abstract} is in this case a Fourier
multiplier, in general unbounded, but densely defined: 
\begin{equation}\label{frame-is-fourier-mult}
\widehat{A_{\sigma} f} = m \cdot \hat{f}
\end{equation}
for $f \in \mbox{dom}(A_{\sigma}) \subseteq L^2(\R)$ a dense subspace, with \emph{symbol}
\begin{equation}\label{definition-of-symbol}
m_\psi(\xi) = \int_{\R} | \hat{\psi}(\beta(\omega)(\xi - \omega)) |^2 \beta(\omega)\,d\omega.
\end{equation}
Thus the admissibility of $\psi \in L^2(\R)$ is equivalent to boundedness and invertibility of the Fourier multiplier $A_{\sigma}$, 
which is in turn equivalent to the existence of constants $A, B$ such that
\begin{equation} \label{eq_star}
0 < A \leq m_\psi(\xi) \leq B < \infty 
\end{equation}
for almost all $\xi \in \R$. (Compare to the more general setting of continuous nonstationary Gabor frames, see \cite{spexxl14}.)

\section{Admissibility} \label{sec:admis0}


We prove a new admissibility condition for the $\alpha$-modulation transform. This generalizes results previously obtained by
Dahlke et al. \cite{DahForRau+2008}. 
More precisely, whereas in \cite{DahForRau+2008} it was shown that band-limited functions, that is functions with compactly supported 
Fourier transform, are admissible, we prove that it suffices to demand just a certain decay of the Fourier transform. 
In particular, we find admissible functions that are compactly supported in time.

Two simple properties of $\beta$, that we will often use without further comment in the sequel, are
\begin{enumerate}[label={(\roman*)}]
\item $\beta$ is symmetric: $\beta(\omega) = \beta(-\omega)$ for all $\omega \in \R$;
\item $\beta$ is bounded: $0 < \frac{1}{1 + |\omega |} \leq \beta(\omega) \leq 1$ for all $\omega \in \R$.
\end{enumerate}

The main result in this section is the following.

\begin{theorem}\label{T:2.1}
Let $0\leq\alpha<1$ and $\psi \in L^2(\R) \setminus \{ 0 \}$ be such that $\hat{\psi}$ is continuous and 
\[
|\hat{\psi}(\xi)| \leq C (1+|\xi|)^{-r}
\]
for all $\xi \in \R$, with 
\[
r > \max\left\{ 1, \frac{\alpha}{2(1-\alpha)} \right\},
\]
then $\psi$ is admissible.
\end{theorem}

\begin{remark} Note that, for $\alpha \to 1$, the exponent $r$ becomes larger and larger: $r \to \infty$.
That means that the closer $\alpha$ is to $1$, the stronger decay of the Fourier transform we need to achieve admissibility. 
\end{remark}

\begin{proof}
We have to show that there exist positive constants $A, B > 0$ such that
\[
0 < A \leq m_\psi(\xi) \leq B < \infty 
\]
for almost all $\xi \in \R$, where
\[
m_\psi(\xi) = \int_{\R} |\hat{\psi}(\beta(\omega)(\xi - \omega))|^2 \beta(\omega)\,d\omega.
\]
For simplicity of notation, set
\[
r_{\xi}(\omega) = \beta(\omega)(\xi - \omega) = \frac{\xi - \omega}{(1 + |\omega |)^{\alpha}},
\]
that means
\[
m_\psi(\xi) = \int_{\R} | \hat{\psi}(r_{\xi}(\omega)) |^2 \beta(\omega)\,d\omega.
\]
First consider $\alpha = 0$. In this case, $\beta(\omega) \equiv 1$, thus
\[
m_\psi(\xi) = \int_{\R} |\hat{\psi}(\xi - \omega)|^2\,d\omega = \| \hat{\psi} \|^2 = \| \psi \|^2
\]
independent of $\xi$, so \eqref{eq_star} is satisfied with $A = B = \| \psi \|^2$. We do not even need the assumptions of 
Theorem \ref{T:2.1}  here; in fact \emph{every} $\psi \in L^2(\R)\backslash\{0\}$ is admissible in the case $\alpha = 0$,
which is a well-known consequence of the orthogonality relation for the short-time Fourier transform.

Now assume $0 < \alpha < 1$. We will first prove three lemmata that yield some simplifications.

\begin{lemma}\label{L:2.2}
We have
\[
m_{\psi}(-\xi) = m_{\overline{\psi}}(\xi),
\]
(where $\overline{\psi}(x) = \overline{\psi(x)}$ denotes the complex conjugate function to $\psi$).
\end{lemma}
\begin{proof}
We have
\[
m_{\psi}(-\xi) = \int_{\R} | \hat{\psi}(r_{-\xi}(\omega)) |^2 \beta(\omega)\,d\omega.
\]
Now,
\[
r_{-\xi}(\omega) = \frac{-\xi - \omega}{(1 + |\omega |)^{\alpha}} = - \frac{\xi - (-\omega)}{(1 + | -\omega |)^{\alpha}} = - r_{\xi}(-\omega),
\]
so
\[
m_{\psi}(-\xi) = \int_{\R} | \hat{\psi}(-r_{\xi}(-\omega)) |^2 \beta(-\omega)\,d\omega = \int_{\R} | \hat{\psi}(-r_{\xi}
(\omega)) |^2 \beta(\omega)\,d\omega.
\]
With $\overline{\hat{\psi}(-\eta)} = \hat{\overline{\psi}}(\eta)$, we conclude
\[
m_{\psi}(-\xi) = \int_{\R} | \hat{\overline{\psi}}(r_{\xi}(\omega)) |^2 \beta(\omega)\,d\omega = m_{\overline{\psi}}(\xi).
\]
\end{proof}
It is clear that if $\psi \in L^2(\R)$ satisfies the assumptions of Theorem \ref{T:2.1}, then $\overline{\psi}$ does so, as well. 
Thus, it suffices to show \eqref{eq_star} only for $\xi \geq 0$.

\begin{lemma}\label{L:2.3}
Let $\hat\psi$ be continuous. The function $m_\psi$ is strictly positive, i.e. $m_\psi(\xi) > 0$ for all $\xi \in \R$.
\end{lemma}
\begin{proof}
It is clear that $m_\psi(\xi) \geq 0$ for all $\xi$.
Suppose there is a $\xi \in \R$ such that $m_\psi(\xi) = 0$. Since the map $\omega \mapsto |\hat{\psi}(r_{\xi}(\omega))|^2 
\beta(\omega)$ is continuous and greater or equal to zero, $\int_{\R} |\hat{\psi}(r_{\xi}(\omega))|^2 \beta(\omega)\, d\omega = 0$ implies 
$|\hat{\psi}(r_{\xi}(\omega))|^2 \beta(\omega) = 0$ for all $\omega \in \R$, which gives
$|\hat{\psi}(r_{\xi}(\omega))| = 0$ for all $\omega \in \R$. Now observe that, for fixed $\xi$,
\[
r_{\xi}(\omega) = \frac{\xi - \omega}{(1 + |\omega |)^{\alpha}} \longrightarrow \begin{cases} -\infty & \quad \mbox{ for }\omega \to + \infty, \\
+ \infty & \quad \mbox{ for } \omega \to - \infty;
\end{cases}
\]
thus the range of $r_{\xi}$ is all of $\R$ and so $\hat{\psi}(\eta) = 0$ for all $\eta \in \R$. But this is equivalent to
$\psi = 0$, a contradiction.
\end{proof}

\begin{lemma}\label{L:2.4}
Let $\psi$ satisfy the assumptions of  Theorem \ref{T:2.1}. The function $m_\psi(\xi)$ is continuous in $\xi$.
\end{lemma}
\begin{proof}
Let $\xi \in \R$ be fixed. Let $\xi^{\prime} \in \R$ with $|\xi^{\prime} - \xi | \leq \frac{1}{2}$. Then
\[
| m_\psi(\xi) - m_\psi(\xi^{\prime}) | \leq \int_{\R} \underbrace{\left| |\hat{\psi}(r_{\xi}(\omega))|^2 - | 
\hat{\psi}(r_{\xi^{\prime}}(\omega)) |^2 \right|\,\beta(\omega)}_{=: I(\omega)}\,d\omega.
\]
We want to use the Dominated Convergence Theorem on this last integral. To this end, consider the integrand $I(\omega)$. If 
$\psi$ satisfies $|\hat{\psi}(\xi)| \leq C (1+|\xi|)^{-r}$, we can estimate
\begin{align*}
I(\omega) & \leq \beta(\omega) \left( C^{2}(1 + | r_{\xi}(\omega) |)^{-2r} + C^{2}(1 + | r_{\xi^{\prime}}(\omega) |)^{-2r}
\right) \\
& = C^{2} \beta(\omega) \left( (1 + | r_{\xi}(\omega) |)^{-2r} + (1 + | r_{\xi^{\prime}}(\omega) |)^{-2r} \right).
\end{align*}
Observe that if $|\xi^{\prime} - \xi | \leq \frac{1}{2}$ then
\begin{align*}
| r_{\xi}(\omega) - r_{\xi^{\prime}}(\omega) | & = |\beta(\omega)(\xi - \omega) - \beta(\omega)(\xi^{\prime} - \omega) | \\
& = \beta(\omega) | \xi - \xi^{\prime} | \leq \frac{1}{2}
\end{align*}
for all $\omega \in \R$. Hence
\[
1 + | r_{\xi^{\prime}}(\omega) | \geq \frac{1}{2} + | r_{\xi}(\omega) |
\]
for all $\omega\in\R$; since, trivially, also
\[ 1 + | r_{\xi}(\omega) | \geq \frac{1}{2} + | r_{\xi}(\omega) |,
\]
we get
\begin{align*}
I(\omega) & \leq C^{2} \beta(\omega) \cdot 2\left(\frac{1}{2} + | r_{\xi}(\omega) |\right)^{-2r} \\
& = \tilde{C} \beta(\omega) \left( \frac{1}{2} + \beta(\omega) | \xi - \omega | \right)^{-2r},
\end{align*}
which is independent of $\xi^{\prime}$ for $|\xi-\xi^\prime|\leq 1/2$. But this last expression is integrable, since, for large $|\omega|$, 
it behaves asymptotically like
\[
\sim |\omega |^{-\alpha} (| \omega |^{-\alpha} |\omega |)^{-2r} = \frac{1}{| \omega |^{\alpha + 2r(1-\alpha)}},
\]
and with $r > 1$ the exponent satisfies
\[
\alpha + 2r(1-\alpha) > \alpha + 2(1 - \alpha) = 2 - \alpha > 1.
\]
It is further clear that for $\xi^{\prime} \to \xi$, we have $r_{\xi^{\prime}}(\omega) \to r_{\xi}(\omega)$ and, 
since $\hat{\psi}$ is continuous, $\hat{\psi}(r_{\xi^{\prime}}(\omega)) \to \hat{\psi}(r_{\xi}(\omega))$, pointwise for 
all $\omega \in \R$. Thus the integrand satisfies
$I(\omega) \to 0$ for $\xi^{\prime} \to \xi$, pointwise for all $\omega \in \R$, so Dominated Convergence finally yields
\[
m_\psi(\xi^{\prime}) - m_\psi(\xi) \to 0
\]
for $\xi^{\prime} \to \xi$, i.e. $m$ is continuous.
\end{proof}

The last two lemmata show that it suffices to prove \eqref{eq_star} only for sufficiently large $\xi$, say $\xi \geq K$; in
this case, on the compact interval $[0, K]$, $m_\psi$ satisfies \eqref{eq_star} as well, since it is continuous and strictly
positive there, so \eqref{eq_star} will hold for all $\xi \geq 0$. We will in fact be able to show an even stronger 
statement: we will prove that $\lim_{\xi \to \infty} m_\psi(\xi) = L > 0$ exists and is positive. From this, the above follows.

Without loss of generality, assume from now on that $\xi > \frac{2}{\alpha}$. Then, obviously, $\alpha \xi > 2$ and $\xi > 2$.

Before we proceed any further, we will discuss the function $r_{\xi}(\omega)$ in more detail.

\begin{lemma}\label{L:2.5}
The derivative of $r_{\xi}(\omega)$ for $\omega \not= 0$ is given by
\[
r^{\prime}_{\xi}(\omega) =  \frac{d}{d\omega}r_{\xi}(\omega) = -\beta(\omega) \left( 1 + \mbox{sgn}(\omega)\cdot \alpha \frac{\xi - \omega}{1 + |\omega |} \right).
\]
\end{lemma}
\begin{proof}
For $\omega > 0$, we have $r_{\xi}(\omega) = \frac{\xi - \omega}{(1 + \omega)^{\alpha}}$, which differentiates to
\begin{align*}
r^{\prime}_{\xi}(\omega) & = \frac{-(1 + \omega)^{\alpha} - (\xi - \omega) \alpha (1 + \omega)^{\alpha - 1}}{(1 + \omega)^{2\alpha}} \\
& = - \frac{1}{(1 + \omega)^{\alpha}} \left( 1 + \alpha \frac{\xi - \omega}{1 + \omega } \right).
\end{align*}
For $\omega < 0$, we have $r_{\xi}(\omega) = \frac{\xi - \omega}{(1 - \omega)^{\alpha}}$, which differentiates to
\begin{align*}
r^{\prime}_{\xi}(\omega) & = \frac{-(1 - \omega)^{\alpha} + (\xi - \omega) \alpha (1 - \omega)^{\alpha - 1}}{(1 - \omega)^{2\alpha}} \\
& = - \frac{1}{(1 - \omega)^{\alpha}} \left( 1 - \alpha \frac{\xi - \omega}{1 - \omega } \right).
\end{align*}
In summary,
\[
r^{\prime}_{\xi}(\omega) = -\frac{1}{(1 + |\omega |)^{\alpha}}\left( 1 + \mbox{sgn}(\omega)\cdot \alpha \frac{\xi - \omega}{1 + |\omega |} \right),
\]
as claimed.
\end{proof}

\begin{lemma}\label{L:2.6}
Set $\omega^{\ast}_{\xi} := \frac{1 - \alpha \xi}{1 - \alpha} < 0$. Then $r_{\xi}(\omega)$ is strictly decreasing
on $]-\infty, \omega^{\ast}_{\xi}[$, strictly increasing on $]\omega^{\ast}_{\xi}, 0[$, and strictly decreasing on
$]0, +\infty[$. At the local minimum $\omega^{\ast}_{\xi}$, we have 
$r_{\xi}(\omega^{\ast}_{\xi}) = \frac{1}{\alpha^{\alpha}} \left( \frac{\xi - 1}{1 - \alpha} \right)^{1 - \alpha} > 0$. 
At the local maximum 0, we have $r_{\xi}(0) = \xi > 0$.
\end{lemma}
\begin{proof}
Let $\omega > 0$. Then 
\begin{align*}
r^{\prime}_{\xi}(\omega) & = -\beta(\omega) \left( 1 + \alpha \frac{\xi - \omega}{1 + \omega} \right) \\
& = - \frac{\beta(\omega)}{1 + \omega}\left( 1 + \omega(1 - \alpha) + \alpha \xi \right) < 0,
\end{align*}
so $r_{\xi}(\omega)$ is strictly decreasing for $\omega > 0$.\\
For $\omega < 0$, note that $r^{\prime}_{\xi}(\omega) = -\beta(\omega) \left(1 - \alpha \frac{\xi - \omega}{1 - \omega}  
\right) = 0$ if and only if $\left(1 - \alpha \frac{\xi - \omega}{1 - \omega}  \right) = 0$ if and only if 
$1 - \omega(1-\alpha) - \alpha \xi = 0$ if and only if $\omega = \frac{1 - \alpha \xi}{1 - \alpha} = \omega^{\ast}_{\xi}$. 
Since $\alpha \xi > 1$, $\omega^{\ast}_{\xi} < 0$. Since $r^{\prime}_{\xi}(\omega)$ is continuous on $]-\infty, 0[$, this is 
the only place where it can change its sign. If $\omega \to - \infty$, then $r^{\prime}_{\xi}(\omega) \uparrow 0$ and 
$r^{\prime}_{\xi}(\omega) < 0$, so $r^{\prime}_{\xi}(\omega) < 0$ for $\omega < \omega^{\ast}_{\xi}$, and $r_{\xi}(\omega)$ 
is decreasing. If $\omega \to 0$, then $r^{\prime}_{\xi}(\omega) \to \alpha \xi - 1 > 0$, so $r^{\prime}_{\xi}(\omega) > 0$ 
for $0 > \omega > \omega^{\ast}_{\xi}$, and $r_{\xi}(\omega)$ is increasing. The values at the local minimum 
$\omega_{\xi}^{\ast}$ and the local maximum $0$ follow by a simple computation.
\end{proof}

Our final lemma will help to compute and estimate several integrals in the following.

\begin{lemma}\label{L:2.7}
Let $I \subseteq \R$ be an interval such that $r_{\xi}$ is monotonous on $I$. Then
\[
\int_I | \hat{\psi}(r_{\xi}(\omega)) |^2 \beta(\omega)\,d\omega = \int_{r_{\xi}(I)} | \hat{\psi}(z) |^2 \frac{1}{| h_{\xi}(r_{\xi}^{-1}(z) |} \,dz,
\]
with
\[
h_{\xi}(\omega) := 1 + \mbox{sgn}(\omega)\cdot \alpha \frac{\xi - \omega}{1 + |\omega |}, \quad \quad \omega \in I.
\]
\end{lemma}
\begin{proof}
We necessarily have $0 \not\in \mbox{int}(I)$, the interior of $I$, since, by Lemma \ref{L:2.6}, $r_{\xi}$ has a local maximum in $0$, 
so $r_{\xi}$ is not monotonous if $0 \in \mbox{int}(I)$.

Observe that, by Lemma \ref{L:2.5},
\begin{equation}\label{h-and-r}
r_{\xi}^{\prime}(\omega) = -\beta(\omega) h_{\xi}(\omega);
\end{equation}
the statement follows from this by the substitution $z = r_{\xi}(\omega)$, $dz = -\beta(\omega) h_{\xi}(\omega)\,d\omega$. 
\end{proof}

We are now ready to finish the proof of Theorem \ref{T:2.1}.

We split the integral defining $m_\psi(\xi)$ into four parts,
\[
m_\psi(\xi) = \int_{\R} |\hat{\psi}(r_{\xi}(\omega)) |^2 \beta(\omega)\, d\omega = \int_{I_1}\ldots + \int_{I_2}\ldots + \int_{I_3}\ldots + \int_{I_4}\ldots,
\]
and treat each part separately. The four intervals are
\begin{align*}
I_1 & = \Big(-\infty,\ \omega_{\xi}^{\ast} - \frac{\alpha \xi^{\alpha}}{2(1 - \alpha)}\Big], \\
I_2 & = \Big[\omega_{\xi}^{\ast} -\frac{\alpha \xi^{\alpha}}{2(1-\alpha)},\ \omega_{\xi}^{\ast} + \frac{\alpha \xi^{\alpha}}{2(1-\alpha)}\Big],\\
I_3 & = \Big[\omega_{\xi}^{\ast} + \frac{\alpha \xi^{\alpha}}{2(1 - \alpha)},\  0\Big], \mbox{ and } \\
I_4 & = [0,\ \infty).
\end{align*}

Observe that
\begin{align*}
\omega_{\xi}^{\ast} + \frac{\alpha \xi^{\alpha}}{2(1 - \alpha)}  = \frac{1 - \alpha \xi}{1 - \alpha} + \frac{\frac{\alpha}{2}
\xi^{\alpha}}{1 - \alpha}  < \frac{1 - \alpha \xi + \frac{\alpha}{2}\xi}{1 - \alpha}  = \frac{1 - \frac{\alpha}{2}\xi}{1 - \alpha} < 0,
\end{align*}
since $\xi > \frac{2}{\alpha}$. Thus $I_2 \subset (-\infty, 0]$, and $I_3$ is well defined.

Note that $r_{\xi}$ is monotonous on $I_1$, $I_3$ and $I_4$.

\begin{itemize}[leftmargin=15pt]

\item $\boldsymbol{I_2} = \Big[\omega_{\xi}^{\ast} - \frac{\alpha \xi^{\alpha}}{2(1 - \alpha)}, \omega_{\xi}^{\ast} + 
\frac{\alpha \xi^{\alpha}}{2(1 - \alpha)}\Big]$: On $I_2$, $r_{\xi}$ has a local minimum at $\omega_{\xi}^{\ast}$. Thus
\begin{align*}
\int_{I_2} & |\hat{\psi}(r_{\xi}(\omega)) |^2 \beta(\omega)\, d\omega \leq \int_{I_2} C^2 (1 + | r_{\xi}(\omega) |)^{-2r} \, d\omega \\
& \leq \int_{I_2} C^2 (1 + | r_{\xi}(\omega_{\xi}^{\ast}) |)^{-2r}\, d\omega \\
& = \frac{\alpha}{1 - \alpha} C^2 \xi^{\alpha} (1 + | r_{\xi}(\omega_{\xi}^{\ast}) |)^{-2r}.
\end{align*}
By Lemma \ref{L:2.6}, we asymptotically have $| r_{\xi}(\omega_{\xi}^{\ast}) | \sim | \xi |^{1 - \alpha}$, so
\[
\int_{I_2} |\hat{\psi}(r_{\xi}(\omega)) |^2 \beta(\omega)\, d\omega \sim |\xi |^{\alpha} (| \xi |^{1 - \alpha})^{-2r} = |\xi |^{\alpha - 2r(1 - \alpha)};
\]
since $r > \frac{\alpha}{2(1-\alpha)}$, we have $\alpha - 2r(1 - \alpha) < 0$, thus
\[
\int_{I_2} |\hat{\psi}(r_{\xi}(\omega)) |^2 \beta(\omega)\, d\omega \to 0
\]
for $\xi \to \infty$.

\item $\boldsymbol{I_1 }= (-\infty, \omega_{\xi}^{\ast} - \frac{\alpha \xi^{\alpha}}{2(1 - \alpha)}\Big]$: We use Lemma \ref{L:2.7} to write
\[
\int_{I_1} | \hat{\psi}(r_{\xi}(\omega)) |^2 \beta(\omega) \, d\omega = \int_{r_{\xi}(I_1)} | \hat{\psi}(z) |^2 \frac{1}{| h_{\xi}(r_{\xi}^{-1}(z) |} \,dz,
\]
where in this case
\[
h_{\xi}(\omega) = 1 - \alpha \frac{\xi - \omega}{1 - \omega }, \quad \quad \omega \in (-\infty,0].
\]
By \eqref{h-and-r}, we have $h_{\xi}(\omega) > 0$ on the interval \\ 
$(-\infty, \omega_{\xi}^{\ast}[$, by Lemma \ref{L:2.6}. Furthermore,
\begin{equation}\label{h-deriv}
h_{\xi}^{\prime}(\omega) = - \alpha \frac{\xi - 1}{(1 - \omega)^2} < 0
\end{equation}
on $(-\infty, 0]$, so $h_{\xi}$ is a monotone decreasing function on $I_1$, and hence $h_{\xi}$ assumes its infimum 
at the rightmost point of $I_1$, i.e. at $\omega_{\xi}^{\ast} - \frac{\alpha \xi^{\alpha}}{2(1 - \alpha)}$. The infimum is given by
\[
\inf_{\omega \in I_1} | h_{\xi}(\omega)| = \left| h_{\xi}\left(\omega_{\xi}^{\ast} - \frac{\alpha \xi^{\alpha}}{2(1 - \alpha)}
\right) \right| = \ldots = \frac{\frac{1-\alpha}{2}\xi^{\alpha}}{\xi - 1 + \frac{1}{2}\xi^{\alpha}},
\]
which behaves asymptotically like $| \xi|^{\alpha - 1}$. We conclude
\[
\frac{1}{| h_{\xi}(\omega) |} \leq \frac{1}{\inf_{\omega \in I_1} | h_{\xi}(\omega) |} \sim | \xi|^{1 - \alpha},
\]
for $\omega \in I_1$. For the transformed interval, we find
$r_{\xi}(I_1)  = [z_1(\xi), + \infty)$ with
\begin{align*}
z_1(\xi) & = r_{\xi}\left(\omega_{\xi}^{\ast} - \frac{\alpha \xi^{\alpha}}{2(1 - \alpha)}\right) \\
& = r_{\xi}\left(\frac{1 - \alpha \xi - \frac{\alpha}{2}\xi^{\alpha}}{1 - \alpha}\right) \\
& = \ldots = \frac{1}{(1-\alpha)^{1 - \alpha}\cdot\alpha^{\alpha}}\cdot \frac{\xi - 1 + \frac{\alpha}{2}\xi^{\alpha}}{(\xi - 1 + \frac{1}{2}\xi^{\alpha})^{\alpha}},
\end{align*}
which also behaves like $| \xi|^{1 - \alpha}$. Putting it all together, we find
\begin{align*}
\int_{r_{\xi}(I_1)} | \hat{\psi}(z) |^2 & \frac{1}{| h_{\xi}(r_{\xi}^{-1}(z) |} \,dz \\
& \leq \frac{1}{\inf_{\omega \in I_1} | h_{\xi}(\omega) |} \int_{r_{\xi}(I_1)} | \hat{\psi}(z) |^2\,dz \\
& \leq \frac{1}{\inf_{\omega \in I_1} | h_{\xi}(\omega)| } \int_{z_1(\xi)}^{\infty} C^2(1 + |z|)^{-2r}\,dz \\
& = \frac{C^2}{(2r - 1)\cdot \inf_{\omega \in I_1} | h_{\xi}(\omega)| } (1 + z_1(\xi))^{-2r + 1},
\end{align*}
that is asymptotically equivalent to $|\xi|^{1-\alpha} |\xi |^{(1-\alpha)(-2r + 1)} = \\ |\xi |^{2(1-\alpha)(1-r)}$. Since $2(1-\alpha)(1 - r) < 0$, 
we finally conclude as $r>1$
\[
\int_{I_1} | \hat{\psi}(r_{\xi}(\omega)) |^2 \beta(\omega) \, d\omega = \int_{r_{\xi}(I_1)} | \hat{\psi}(z) |^2  \frac{1}{| 
h_{\xi}(r_{\xi}^{-1}(z)) |} \,dz \to 0\]
for $\xi \to \infty$.

\item $\boldsymbol{I_3} = \big[\omega_{\xi}^{\ast} + \frac{\alpha \xi^{\alpha}}{2(1 - \alpha)}, 0\big]$: This is very similar to the previous case $I_1$. We have
\[
\int_{I_3} | \hat{\psi}(r_{\xi}(\omega)) |^2 \beta(\omega) \, d\omega = \int_{z_2(\xi)}^{\xi} | \hat{\psi}(z) |^2 \frac{1}{| h_{\xi}(r_{\xi}^{-1}(z) |} \,dz
\]
with $h_{\xi}$ as above, and
\[
z_2(\xi) = \frac{1}{(1-\alpha)^{1 - \alpha}\cdot\alpha^{\alpha}} \cdot\frac{\xi - 1 - \frac{\alpha}{2}\xi^{\alpha}}{(\xi - 1 - \frac{1}{2}\xi^{\alpha})^{\alpha}}.
\]
On $I_3$, $h_{\xi}(\omega) < 0$ and $h^{\prime}_{\xi}(\omega) < 0$ by \eqref{h-deriv}, so
\[
\inf_{\omega \in I_3} | h_{\xi}(\omega)| = | h_{\xi}(\omega_{\xi}^{\ast} + \frac{\alpha \xi^{\alpha}}{2(1 - \alpha)}) | 
= \ldots = \frac{\frac{1-\alpha}{2}\xi^{\alpha}}{\xi - 1 - \frac{1}{2}\xi^{\alpha}} \sim |\xi |^{\alpha - 1}.
\]
This yields
\begin{align*}
\int_{I_3} | \hat{\psi}(r_{\xi}(\omega)) & |^2 \beta(\omega) \, d\omega \\
& \leq \frac{1}{\inf_{\omega \in I_3} | h_{\xi}(\omega)| } \int_{z_2(\xi)}^{\xi} C^2(1 + |z|)^{-2r}\,dz \\
& \sim |\xi |^{2(1-\alpha)(1-r)},
\end{align*}
and this goes to $0$ for $\xi \to \infty$ because $2(1-\alpha)(1 - r) < 0$.

\item $\boldsymbol{I_4} = [0, \infty)$: We consider
\[
\int_{I_4} | \hat{\psi}(r_{\xi}(\omega)) |^2 \beta(\omega)\,d\omega = \int_{r_{\xi}(I_4)} | \hat{\psi}(z) |^2 \frac{1}{| h_{\xi}(r_{\xi}^{-1}(z) |} \,dz,
\]
now with
\[
h_{\xi}(\omega) = 1 + \alpha \frac{\xi - \omega}{1 + \omega }, \quad \quad \omega \in I_4.
\]
Since $r_{\xi}(0) = \xi$ and $\lim_{\omega \to \infty} r_{\xi}(\omega) = -\infty$, we have $r_{\xi}(I_4) = (-\infty, \xi]$. Let $\varepsilon > 0$ be given. Choose $A > 0$ such that
\[
\int_{\R \setminus [-A, A]} | \hat{\psi}(z) |^2\,dz \leq \varepsilon,
\]
that means
\[
\left| \int_{[-A, A]} | \hat{\psi}(z) |^2\,dz - \| \psi \|^2 \right| \leq \varepsilon.
\]
Now assume $\xi > A$. Then
\[
\int_{-\infty}^{\xi} | \hat{\psi}(z) |^2 \frac{1}{| h_{\xi}(r_{\xi}^{-1}(z) |} \,dz  = \int_{[-A, A]}\ldots + \int_{(-\infty, \xi] \setminus [-A, A]}\ldots.
\]
The second integral can be estimated as follows: first observe that
\[
h_{\xi}(\omega) = \frac{1 + \alpha \xi + \omega(1 - \alpha)}{1 + \omega} > 0
\]
on $I_4$. Its derivative is
\[
h_{\xi}^{\prime}(\omega) = - \alpha \frac{\xi + 1}{(1 + \omega)^2} < 0,
\]
so $\inf_{\omega \in I_4} |h_{\xi}(\omega) | = \lim_{\omega \to \infty} | h_{\xi}(\omega) | = 1 - \alpha$. Hence
\[
\int_{(-\infty, \xi] \setminus [-A, A]}\ldots \leq \frac{1}{1 - \alpha} \int_{(-\infty, \xi] \setminus [-A, A]} | \hat{\psi}(z) |^2\,dz \leq \frac{\varepsilon}{1 - \alpha}.
\]
For the first integral, we use that for every fixed $A > 0$,
\[
\lim_{\xi \to \infty} h_{\xi}(r_{\xi}^{-1}(z)) = 1
\]
uniformly on $[-A, A]$. This result can be found 
in \cite[Lemma 5.1 and the proof of Theorem 5.2]{DahForRau+2008}.
Hence, we obtain that
\[
\int_{[-A, A]}\ldots \to \int_{[-A, A]} |\hat{\psi}(z)|^2\,dz
\]
for $\xi \to \infty$. Thus
\[
\left| \int_{I_4} | \hat{\psi}(r_{\xi}(\omega)) |^2 \beta(\omega)\,d\omega - \| \psi \|^2 \right| \leq \varepsilon + 
\frac{\varepsilon}{1 - \alpha},
\]
for $\xi$ sufficiently big.
Since $\varepsilon$ was arbitrary, we conclude
\[
\int_{I_4} | \hat{\psi}(r_{\xi}(\omega)) |^2 \beta(\omega)\,d\omega \to \| \psi \|^2
\]
for $\xi \to \infty$.

\end{itemize}

All in all, we have thus shown
\[
m_\psi(\xi) = \int_{I_1}\ldots + \int_{I_2}\ldots + \int_{I_3}\ldots + \int_{I_4}\ldots \longrightarrow \|\psi\|^2
\]
for $\xi \to \infty$, which finally concludes the proof of Theorem \ref{T:2.1}.
\end{proof}


\begin{remark}
As already stated in the introductory part, it is a major objective of this paper to prove the existence of compactly
supported admissible windows for the $\alpha$-modulation transform. The assumptions of Theorem \ref{T:2.1} are in particular 
satisfied for $\psi \in \mathcal{S}(\R) \subseteq L^2(\R)$, the Schwartz class 
of infinitely differentiable rapidly decaying functions (their Fourier transforms are again of the same class, thus decay 
faster than 
any given polynomial). Since there exist Schwartz functions with compact support, the  existence of compactly supported 
admissible functions is guaranteed by Theorem \ref{T:2.1}. 
\end{remark}


\section{Generalized coorbit theory}\label{sec:coorbit_basics}

In this section we briefly introduce the concept of generalized coorbit theory. We will, however, only present the bare 
necessities from \cite{DahForRau+2008} and \cite{FoRa2005}
to grasp the underlying idea and motivate the calculations in Section \ref{sec:coorbit_alpha}. For further reading on coorbit theory we 
refer the interested reader to \cite{FeiGro_88,FeiGro_89,FeiGro_89-2} (classical coorbit theory) and 
\cite{DahForRau+2008,For2007,FoRa2005} (generalized coorbit theory).

\subsection{Construction of generalized coorbit spaces}\label{subsec:generalized_coorbit}

The fundamental idea behind coorbit theory is that features of a function, like smoothness or decay,  manifest 
in the behavior of its voice transform. Hence, one constructs coorbit spaces as those functions/distributions whose voice transform
belongs to a certain Banach space. In the present paper we will focus on weighted Lebesgue spaces  \cite{DahForRau+2008}. However, more general
Banach spaces may be used with some modifications,
see \cite{FoRa2005}.

In this section we consider the same general setting as in Section \ref{sec:prel0}, i.e. let $X$ be a homogeneous space, 
$\sigma$ a section from $X$ to $G$ and
$\pi$ a unitary group representation.
Let $v\geq1$ be measurable and $1\leq p\leq\infty$ and define
\[
L^p_v(X):=\big\{F\mbox{ measurable, }Fv\in L^p(X)\big\}.
\]
equipped with the natural norm $\|F\|_{L^p_v}:=\|Fv\|_{L^p}$.
Let $\psi$ be admissible and define the reproducing kernel $R$ by
$$
\mathcal{R}(x,y):=\langle A^{-1}_\sigma\pi(\sigma(x))\psi,\pi(\sigma(y))\psi\rangle.
$$
$\mathcal{R}$ reproduces the image of the voice transforms, i.e. for $f\in\H$ and $\forall\ x\in X$ it holds
\begin{equation}\label{rep-formulae1}
V_\psi f(x)=\int_X V_\psi f(y)\mathcal{R}(y,x)d\mu(y),
\end{equation}
\begin{equation}\label{rep-formulae2} 
W_\psi f(x)=\int_X W_\psi f(y)\mathcal{R}(y,x)d\mu(y).
\end{equation}
Moreover, we define the weight function
$$
w(x,y):=\max\left\{\frac{v(x)}{v(y)},\frac{v(y)}{v(x)}\right\}.
$$
The following condition is fundamental for establishing generalized coorbit theory. 
\begin{equation}\label{fundamental-condi}
 \rho:=\mbox{ess}\sup_{y\in X}\int_X |\mathcal{R}(x,y)|w(x,y)d\mu(x)<\infty
\end{equation}
Throughout the rest of this section we will assume that \eqref{fundamental-condi} holds.

We define the reservoir spaces $\H_{1,v}$ and  $\K_{1,v}$ by
\begin{align*}
 &\H_{1,v}:=\{f\in \H:\ W_\psi f\in L^1_v(X)\}\\
  &\K_{1,v}:=\{f\in \H:\ V_\psi f\in L^1_v(X)\},
\end{align*}
with norms $\|f\|_{\H_{1,v}}:=\|W_\psi f\|_{L^1_v}$ and $\|f\|_{\K_{1,v}}:=\|V_\psi f\|_{L^1_v}$. 
The fundamental condition \eqref{fundamental-condi} guarantees that $\pi(\sigma(x))\psi\in \H_{1,v}$ and 
$A_\sigma^{-1}\pi(\sigma(x))\in\K_{1,v}$, $\forall x\in X$. Consequently, both $\H_{1,v}$ and $\K_{1,v}$ are dense in $\H$
and the embedding is continuous. Moreover, the spaces are complete, i.e. Banach spaces.

Now, introduce the anti dual spaces $\H_{1,v}'$ and $\K_{1,v}'$ (the space of all bounded and conjugate linear functionals on $\H_{1,v}$ and 
$\K_{1,v}$ respectively), then
\begin{align*}
 &\H_{1,v}\subset \H\subset \H_{1,v}'\\
  &\K_{1,v}\subset \H\subset \K_{1,v}'.
\end{align*}
It can be shown that $\H_{1,v}$ is norm dense in $\H$ and weak-$\ast$ dense in $\H_{1,v}'$. 
The operators $V_\psi$ and $W_\psi$ can be extended to $\H_{1,v}'$ and $\K_{1,v}'$, 
respectively, by setting
\begin{align*}
 &V_{\psi}f(x):=\langle f,\pi(\sigma(x))\psi\rangle_{\H_{1,v}'\times\H_{1,v}}\\
  &W_{\psi}f(x):=\langle f,A_\sigma^{-1}\pi(\sigma(x))\psi\rangle_{\K_{1,v}'\times\K_{1,v}}.
\end{align*}

For $1\leq p\leq\infty$, the generalized coorbit spaces $\H_{p,v}$ and $\H_{p,v}$ may be defined as
\setlength\jot{0.3cm}
\begin{align*}
 &\H_{p,v}:=\{f\in \K_{1,v}':\ W_\psi f\in L^p_v(X)\},\\ 
  &\K_{p,v}:=\{f\in \H_{1,v}':\ V_\psi f\in L^p_v(X)\},
\end{align*}
with norms $\|f\|_{\H_{p,v}}:=\|W_\psi f\|_{L^p_v}$ and $\|f\|_{\K_{p,v}}:=\|V_\psi f\|_{L^p_v}$.

\begin{theorem}
Let $\psi\in \H$ be admissible, such that \eqref{fundamental-condi} is satisfied, then both $\H_{p,v}$ and 
$\K_{p,v}$ are Banach spaces and the reproducing formulas  \eqref{rep-formulae1} and \eqref{rep-formulae2} extend to $\H_{p,v}$ and 
$\K_{p,v}$.
\end{theorem}
The treatment of two spaces $\H_{p,v}$ and $\K_{p,v}$ is somewhat cumbersome. However, one can show that they coincide 
if the frame $\{\pi(\sigma(x))\psi\}_{x\in X}$ is intrinsically localized, (see  \cite{fogr05,FoRa2005}):

Define $K_{\psi,\varphi}^\kappa(x,y):=\langle A_\sigma^{-\kappa}\pi(\sigma(x))\psi, \pi(\sigma(y))\varphi\rangle$, $\kappa\in\Z$ and set
$K_{\psi}^\kappa:=K_{\psi,\psi}^\kappa$.
If 
\begin{equation}\label{kernel-integrable0}
\mbox{ess}\sup_{y\in X}\int_X |K^0_{\psi}(x,y)|w(x,y)d\mu(x)<\infty,
\end{equation}
and
\begin{equation}\label{kernel-integrable2}
\mbox{ess}\sup_{y\in X}\int_X |K^2_{\psi}(x,y)|w(x,y)d\mu(x)<\infty,
\end{equation}
 then $\H_{p,v}=\K_{p,v}$ for all $p\in[1,\infty]$ with equivalent norms. Moreover, if \eqref{kernel-integrable0} and \eqref{kernel-integrable2} hold
for both $\psi$  and $\varphi$ and if in addition 
\begin{align*}
\max &\left\{
\mbox{ess}\sup_{x\in X}\int_X |K^0_{\psi,\varphi}(x,y)|w(x,y)d\mu(y),\right.\\
&\quad \left.\mbox{ess}\sup_{y\in X}\int_X |K^0_{\psi,\varphi}(x,y)|w(x,y)d\mu(x)\right\}<\infty,
\end{align*}
then $\H_{p,v}$ does not depend on whether it is generated by $\psi$ or by $\varphi$.

Observe that the fundamental condition \eqref{fundamental-condi} amounts to a statement of the integrability of $K_\psi^1$.

\subsection{Discretization in generalized coorbit spaces}\label{subsec:discrete_coorbit}

Throughout this section we mainly present the results from \cite{FoRa2005}
with a minor but crucial modification in the definition of the local oscillations kernel introduced in \cite{howixxl14}.

\begin{definition}\label{def-admissible-covering}
Let $\mathcal{I}$ be a countable index set.
A family $\mathcal{U} = \{U_i\}_{i\in\mathcal{I}}$ of relatively compact subsets of $X$ with non-empty interior is called an 
\textbf{admissible covering} of $X$ if the following conditions are satisfied:
\begin{enumerate}[label={(\roman*)}]
	\item Covering property: $X = \bigcup_{i\in\mathcal{I}}U_i$,
	\vspace{0.3cm}
  \item Finite overlap:\hspace{0.7cm} $\sup_{j\in\mathcal{I}}\#\{i\in\mathcal{I}:\ U_i\cap U_j\neq \emptyset\} \leq N < \infty$.
\end{enumerate}

If, moreover, $\mu(U_i) \geq A > 0$ for all $i \in \mathcal{I}$ and there exists a constant $C>0$, such that $\mu(U_i) \leq C \mu(U_j)$
for all $i,j$ with $U_i \cap U_j \neq \emptyset$, then $\mathcal{U}$ is called \textbf{moderate}.
\end{definition}

\begin{definition}
 Let $\big(B,\norm{B}{\cdot}\big)$ be a Banach space. A family $\{g_i\}_{i\in\mathcal{I}}\subset B$ is called an \textbf{atomic decomposition} if 
 there exists a 
 BK-space $\big(B^\natural,\norm{B^\natural}{\cdot}\big)$ and a family of bounded linear functionals 
 $\{\lambda_i\}_{i\in\mathcal{I}}\subset B^\ast$ such that
 \begin{enumerate}[label={(\roman*)}]
  \item If $f\in B$, then $\{\lambda_i(f)\}_{i\in\mathcal{I}}\in B^\natural$ and there exists $M>0$ such that 
  $$
  \norm{B^\natural}{\{\lambda_i(f)\}_{i\in\mathcal{I}}}\leq M\norm{B}{f},\ \mbox{for all}\ f\in B
  $$
  \item If $\{\nu_i\}_{i\in\mathcal{I}}\in B^\natural$, then $f=\sum_{i\in\mathcal{I}}\nu_i g_i\in B$ 
  (with unconditional convergence in some suitable topology) and there exists $m>0$ such that
   $$
  m\norm{B}{f}\leq \norm{B^\natural}{\{\nu_i\}_{i\in\mathcal{I}}}
  $$
  \item $f=\sum_{i\in\mathcal{I}}\lambda_i(f)g_i,\ \mbox{for all}\ f\in B$
 \end{enumerate}
 A family $\{h_i\}_{i\in\mathcal{I}}\subset B^\ast$ is called a \textbf{Banach frame} if there exists a 
 BK-space $\big(B^\flat,\norm{B^\flat}{\cdot}\big)$ and a bounded linear reconstruction operator $\Omega:B^\flat\to B$ such that
 \begin{enumerate}[label={(\roman*)}]
  \item If $f\in B$, then $\{h_i(f)\}_{i\in\mathcal{I}}\in B^\flat$ and there exists $m,M>0$ such that
  $$
 m\norm{B}{f}\leq\norm{B^\flat}{\{h_i(f)\}_{i\in\mathcal{I}}}\leq M\norm{B}{f},\ \mbox{for all}\ f\in B
  $$
  \item $f=\Omega\big(\{h_i(f)\}_{i\in\mathcal{I}}\big),\ \mbox{for all}\ f\in B$
 \end{enumerate}

 \end{definition}
 
 The generalized {local oscillations kernel} with respect to the moderate admissible covering $\mathcal{U}$ is defined as in 
 \cite{howixxl14} by
\begin{align}\label{loc-osc-defined}
osc_{\mathcal{U},\Gamma}(x,y)	& := \sup_{z\in Q_y}|\langle A^{-1}_\sigma\pi(\sigma(x))\psi,\pi(\sigma(y))\psi-\Gamma(y,z)
\pi(\sigma(z))\psi\rangle| \notag\\
& = \sup_{z\in Q_y}|\mathcal{R}(x,y)-\Gamma(y,z)\mathcal{R}(x,z)|
\end{align}
where $\Gamma:X\times X\rightarrow \C$ is measurable and satisfies $|\Gamma| \equiv 1$, $Q_y := \bigcup_{i \in \mathcal{I}(y)}U_i$ and 
$\mathcal{I}(y) := \{i \in \mathcal{I}: y \in U_i\}$.

Define $\gamma=\max\{\gamma_1,\gamma_2\}$, with
\begin{equation}\label{gamma1}
\gamma_1:=\mbox{ess}\sup_{x\in X}\int_X |osc_{\mathcal{U},\Gamma}(x,y)|w(x,y)d\mu(y)
\end{equation}
and
\begin{equation}\label{gamma2}
\gamma_2:=\mbox{ess}\sup_{y\in X}\int_X |osc_{\mathcal{U},\Gamma}(x,y)|w(x,y)d\mu(x).
\end{equation}
Moreover, we need the following technical condition
$$C_{w,\mathcal{U}}:=\sup_{i\in\mathcal{I}}\sup_{x,y\in\ U_i}w(x,y)<\infty.$$
We are now able to formulate the following discretization result given in \cite{howixxl14}:

\begin{theorem}\label{discretization_result}
Let $\mathcal{U}$ be a moderate admissible covering, such that
\begin{equation}\label{constant-smaller-one}
 \gamma\cdot\Big(\rho+\max\big\{\rho\cdot C_{w,\mathcal{U}},\ \rho+\gamma\big\}\Big)<1,
\end{equation}
 then $\big\{\pi(\sigma(x_i))\psi\big\}_{i\in\mathcal{I}}$ is a Banach frame and an atomic decomposition for $\H_{p,v}$ where
 $x_i\in U_i$ for all
 $i\in\mathcal{I}$ can be chosen arbitrarily.
\end{theorem}

\begin{remark}
We omit here the details about the suitable choice of the BK-spaces. The information can be found in \cite[Section 5.1]{FoRa2005}.
\end{remark}
\begin{remark}
Observe that this discretization scheme is very powerful as, once the technical condition \eqref{constant-smaller-one} is 
checked, it ensures that $\big\{\pi(\sigma(x_i))\psi\big\}_{i\in\mathcal{I}}$
is a Banach frame and an atomic decomposition for \textbf{all} coorbit spaces $\H_{p,v}$ simultaneously.
Typically, the strategy to ensure \eqref{constant-smaller-one} is to construct a sequence of moderate admissible coverings $\mathcal{U}^n$,
such that $\gamma_n\rightarrow 0$ and 
$C_{w,\mathcal{U}^n}$ is uniformly bounded.
\end{remark}


\section{Generalized coorbit theory for the \texorpdfstring{$\alpha$}{[alpha]}-transform}\label{sec:coorbit_alpha}


The generalized coorbit spaces for the $\alpha$-modulation transform can be identified with $\alpha$-modulation spaces,
see \cite{DahForRau+2008}.
These spaces were introduced independently by Gr{\"o}bner \cite{FeiGr_85,Gr_92} and  P{\"a}iv{\"a}rinta/Somersalo \cite{paiso88} as an 
``intermediate'' family of Banach spaces between modulation spaces and homogeneous Besov spaces,
 the smoothness spaces 
associated
to the short-time Fourier transform and the continuous wavelet transform respectively.
For further reading on $\alpha$-modulation  spaces, see for example  \cite{bo04,For2007,hawa14,hona03}.

\subsection{Integrability of the kernels \texorpdfstring{$K_{\psi,\varphi}^\kappa$}{K\_{[psi],[varphi]}\textasciicircum[kappa]}}\label{subsec:integrability_kernel}


In this section we will apply the results from Section \ref{subsec:generalized_coorbit} to the $\alpha$-modulation transform. In particular, we
will
 prove that, for a certain class of window functions, (a) the fundamental condition \eqref{fundamental-condi} is satisfied and (b)
 the coorbit spaces $\H_{p,v}$ and $\K_{p,v}$ are equal and independent of the particular choice of the window.
 
We consider polynomial weight functions in the frequency variable
\[
v_s(x,\omega) := v_s(\omega) := (1+\left|\omega\right|)^s,\quad s \in \R,
\]
and, consequently,
\begin{align*}
w_s(x,\omega,x^\ast,\omega^\ast) := w_s(\omega,\omega^\ast) & := \max\left\{\frac{v_s(\omega^\ast)}{v_s(\omega)},\frac{v_s(\omega)}{v_s(\omega^\ast)}\right\} \\
& = \max\left\{\left(\frac{1+\left|\omega^\ast\right|}{1+\left|\omega\right|}\right),\left(\frac{1+\left|\omega\right|}{1+\left|\omega^\ast\right|}\right)
\right\}^{\left|s\right|}.
\end{align*}

The kernel $K_{\psi,\varphi}^\kappa$ is of the following shape
\[
K_{\psi,\varphi}^\kappa(x,\omega,x^\ast,\omega^\ast) = 
\langle T_x  M_\omega D_{\beta(\omega)}\psi,A^{-\kappa}_\sigma T_{x^\ast}M_{\omega^\ast} D_{\beta(\omega^\ast)}\varphi\rangle.
\]

\begin{theorem}\label{kern-integr-th}
Let  $s\geq 0$ and
$\psi,\varphi \in L^2(\R)$, such that 
$\hat\psi,\hat\varphi \in C^2(\R)$, with $|\hat\psi^{(l)}(\xi)|\leq C (1+|\xi|)^{-r}$, for $l = 0,1,2$ (and the 
same decay requirements are also imposed on $\varphi$), where the parameter $r$ 
is chosen such that
\begin{equation}\label{decay-constant1}
r > \frac{2+2s+7\alpha-4\alpha^2}{2(1-\alpha)^2}.
\end{equation}
Then, for $\kappa = 0,1,2$,
\begin{equation}\label{finite-integral-kernels}
 \sup\limits_{(x^\ast,\omega^\ast)\in\R^2}\int_{\R}\int_{\R}\left|K_{\psi,\varphi}^\kappa(x,\omega,x^\ast,\omega^\ast)
 \right|w_s(\omega,\omega^\ast)\;dxd\omega < \infty.
\end{equation}
\end{theorem}
In the course of proving this theorem, we need several auxiliary results. Therefore, we give a short sketch of 
the proof first to motivate the lemmata.

\textit{Idea of the Proof of Theorem \ref{kern-integr-th}.}
Roughly speaking, the proof depends on two main ideas. First, one rearranges the kernel $K_{\psi,\varphi}$ and observes that 
(after a change of variables)
the integral with respect to $x$ may be rewritten as 
$$\int_\R|\mathcal{F}(G_{\omega,\omega^\ast})(x)|dx,$$
for some function $G_{\omega,\omega^\ast}$ depending on $\omega$ and $\omega^\ast$.
Then basic Fourier theory yields that, given certain regularity, the following estimates hold pointwise
$$|\mathcal{F}(f)(\xi)|\leq \|f\|_1, \ \ \mbox{ and }\ \ |\mathcal{F}(f)(\xi)|\leq \frac{\|f^{(2)}\|_1}{4\pi^2\xi^2}.$$
Consequently,
\begin{align}\label{pointwise-bound-G}
 |\mathcal{F}(G_{\omega,\omega^\ast})(x)| & \leq \min\left\{\|G_{\omega,\omega^\ast}\|_1,
 \frac{\|G_{\omega,\omega^\ast}^{(2)}\|_1}{4\pi^2 x^2}\right\} \notag\\
& \leq C \max\left\{\|G_{\omega,\omega^\ast}\|_1,
\|G_{\omega,\omega^\ast}^{(2)}\|_1\right\}\min\left\{1,\frac{1}{ x^2}\right\}.
\end{align}
which guarantees integrability with respect to $x$ if the $L^1$-norms are finite.\\ Second, the auxiliary results 
 Lemma \ref{weightlemma} to Lemma \ref{Lambda-bounded} provide  pointwise estimates of the weight $m_s$ (after substitution in $\omega$)
 and $G_{\omega,\omega^\ast}^{(k)}$. Lemma \ref{shearlemma} helps to show that the 
 $L^1$-norm of $\|G_{\omega,\omega^\ast}^{(k)}\|_1$ with respect to $\omega$ is finite (independently of $\omega^\ast$).\\

We will use the following results:
\begin{lemma}\label{weightlemma}
It holds
\[
  \widetilde{w}_s(\omega,\omega^\ast) := w_s(\omega^\ast+\beta(\omega^\ast)^{-1}\omega,\omega^\ast) \leq (1+|\omega|)^{\frac{|s|}{1-\alpha}}
\]
for all $\omega,\omega^\ast,s \in \R$. 
\end{lemma}
(For the proof, see  \cite[Lemma 5.8]{DahForRau+2008}.)

We  now show that the derivatives of the  symbol $m_\psi$, corresponding to the Fourier multiplier  $A_\sigma$
defined in \eqref{definition-of-symbol}, are 
polynomially decaying.
 
\begin{lemma}\label{symbol-derivatives-estimate}
Let $\psi\in L^2(\R)$ be such that $\hat\psi\in C^k(\R)$ for some $k\in\mathbb{N}$, and $|\hat\psi^{(l)}(\xi)|\leq C(1+|\xi|)^{-r}$, 
for all $l = 0,1,...,k$ and $r>\max\big\{1,\frac{\alpha}{2(1-\alpha)}\big\}$.\\
Then, $m_\psi\in C^k(\R)$ and for all $l = 0,1,...,k$,
\[
|m^{(l)}_\psi(\xi)|\leq C(1+|\xi|)^{-\alpha l}.
\]
\end{lemma}
\begin{proof}
The case $l=0$ follows from the fact that a function with the mentioned properties is admissible.
 
We present a detailed proof for the case $l=1$.

By the Mean Value Theorem, we have, for some $\eta$ between $\xi$ and $\xi + \varepsilon$,
\begin{align*}
m_\psi'(\xi) & = \lim_{\varepsilon\rightarrow 0}\frac{m_\psi(\xi+\varepsilon)-m_\psi(\xi)}{\varepsilon} \\
& =  \lim_{\varepsilon\rightarrow 0}\int_\R \frac{1}{\varepsilon}\Big( |\hat\psi(\beta(\omega)(\xi+\varepsilon-\omega))|^2-
  |\hat\psi(\beta(\omega)(\xi-\omega))|^2\Big)\beta(\omega)\;d\omega \\
& =  \lim_{\varepsilon\rightarrow 0}\int_\R 
  2\,\mbox{Re} (\hat\psi'\overline{\hat\psi})(\beta(\omega)(\eta-\omega))\beta^2(\omega)\;d\omega.
\end{align*}
It is not difficult to see that 
\[
 \sup_{\delta\in[-1,1]}2\,C (1+\beta(\omega)(\xi+\delta-\omega))^{-2r}\beta^2(\omega)
\]
is an integrable majorant as it asymptotically behaves like $\sim |\omega|^{-2r(1-\alpha)-2\alpha}$.
 The Dominated Convergence Theorem thus yields 
\[
  m_\psi'(\xi)=\int_\mathbb{R} 2Re(\hat\psi'\overline{\hat\psi})(\beta(\omega)(\xi-\omega))\beta^2(\omega)d\omega.
\]
Let us first assume that $\xi \geq 0$ in the sequel.
\begin{align*}
 |m_\psi'(\xi)| &= \left| \int_\R 2\,\mbox{Re} (\hat\psi'\overline{\hat\psi})(\beta(\omega)(\xi-\omega) \beta^2(\omega)\;d\omega \right | \\
 & \leq \,C \int_\R (1+\beta(\omega)|\xi-\omega|)^{-2r} \beta^2(\omega)\;d\omega \\
 & = \,C\int_{-\infty}^{\xi/2}((1+|\omega|)^\alpha+|\xi-\omega|)^{-2r}(1+|\omega|)^{2\alpha(r-1)}\;d\omega \\
 & \quad \quad + \,C \int_{\xi/2}^\infty (1+\beta(\omega)|\xi-\omega|)^{-2r} \beta^2(\omega)\;d\omega.
 \end{align*}
Since $2\alpha(r-1) \geq 0$, it follows that $(1+|\omega|)^{2\alpha(r-1)} \leq (1+|\xi-\omega|)^{2\alpha(r-1)}$, for $\omega\in (-\infty,\xi/2]$. 
Hence,
\begin{align*}
|m_\psi'(\xi)| & \leq \,C \int_{-\infty}^{\xi/2}(1+|\xi-\omega|)^{-2r}(1+|\xi-\omega|)^{2\alpha(r-1)}\;d\omega \\
& \quad \quad + \,C \beta(\xi/2)\int_{\xi/2}^\infty  (1+\beta(\omega)|\xi-\omega|)^{-2r} \beta(\omega)\;d\omega \\
& \leq \,C \int_{-\infty}^{-\xi/2}(1+|\omega|)^{-2(1-\alpha)r-2\alpha}\;d\omega \\
& \quad \quad + \,C \beta(\xi/2)\int_\R  (1+\beta(\omega)|\xi-\omega|)^{-2r} \beta(\omega)\;d\omega \\
& \leq C (1+|\xi/2|)^{-2(1-\alpha)r-2\alpha+1} + C (1+|\xi/2|)^{-\alpha}\\
& \leq C (1+|\xi|)^{-2(1-\alpha)r-2\alpha+1} + C (1+|\xi|)^{-\alpha}.
\end{align*}
 The second last inequality follows if we observe that the second integral is just the symbol $m_\phi$, where the function $\phi$ is defined 
 via its Fourier transform by  $\hat\phi(\xi) = (1+|\xi|)^{-r} \in C(\R)$. Theorem \ref{T:2.1} then yields that this 
 term is bounded. As $2(1-\alpha)r+2\alpha-1>\alpha$ whenever $r>1/2$ it follows that
\[
|m_\psi'(\xi)| \leq C (1+|\xi|)^{-\alpha},\ \forall \xi\geq0.
\]
Now if $\xi\leq0$ we observe that  $(1+|\omega|)^{2\alpha(r-1)} \leq (1+|\xi-\omega|)^{2\alpha(r-1)}$, for $\omega\in [\xi/2,\infty)$,
and use the same arguments as above.

For higher derivatives one proceeds iteratively and obtains the condition $2(1-\alpha)r+(l+1)\alpha-1>l\alpha$ which is again
satisfied whenever $r>1/2$.
\end{proof}

\begin{corollary}\label{estimate-symbol-term}
Let $\psi\in L^2(\R)$ be such that $\hat\psi\in C^2(\R)$ and $|\hat\psi^{(l)}(\xi)| \leq C (1+|\xi|)^{-r}$, 
for all $l =0,1,2$, and $r>\max\big\{1,\frac{\alpha}{2(1-\alpha)}\big\}$.
Define 
\[
h_{\omega,\kappa}(\xi) := m^{-\kappa}_\psi(\beta^{-1}(\omega) \xi + \omega).
\]
Then $h_{\omega,\kappa}\in C^2(\R)$ and 
\begin{equation}\label{eq-estimate}
|h_{\omega,\kappa}^{(l)}(\xi)|\leq C \left(\frac{1+|\omega|}{1+|\beta^{-1}(\omega)\xi+\omega|}\right)^{\alpha l},
\end{equation}
for $l = 0,1,2$.
\end{corollary}
\begin{proof}
We have
\[
h_{\omega,\kappa}'(\xi) = -\kappa m^{-(\kappa+1)}_\psi(\beta^{-1}(\omega)\xi+\omega)m_\psi'(\beta^{-1}(\omega)\xi+\omega)\beta^{-1}(\omega)
\]
as well as 
\begin{align*}
h_{\omega,\kappa}''(\xi) & = \kappa(\kappa+1) m_\psi^{-(\kappa+2)}(\beta^{-1}(\omega)\xi+\omega) m_\psi'(\beta^{-1}(\omega)\xi+\omega)^2\beta^{-2}(\omega)
\\
& \quad - \kappa m_\psi^{-(\kappa+1)}(\beta^{-1}(\omega)\xi+\omega)m_\psi''(\beta^{-1}(\omega)\xi+\omega)\beta^{-2}(\omega).
\end{align*}

Lemma \ref{symbol-derivatives-estimate}, together with $m_\psi(\xi)\geq A$, for a.e. $\xi$, therefore yield the result.
\end{proof}

Next, define the function
\begin{equation}\label{def-Lambda}
\Lambda(\xi,\omega) := \frac{1+|\omega|}{(1+|\xi|)^{1/(1-\alpha)}(1+|\beta^{-1}(\omega)\xi+\omega|)},
\end{equation}
for $\xi, \omega \in \R$.

\begin{lemma}\label{Lambda-bounded}
The function $\Lambda$ is bounded from above, precisely 
\[
\sup_{\xi,\omega \in \R} \Lambda(\xi,\omega) \leq  2^{1/(1-\alpha)}.
\]
\end{lemma}
\begin{proof}
Since $\Lambda(-\xi,\omega) = \Lambda(\xi,-\omega)$, we may assume that $\xi\geq 0$. Moreover, observe that if $\omega \geq 0$ we 
have $\Lambda(\xi,\omega) \leq \Lambda(\xi,-\omega)$, for all $\xi \geq 0$. Hence, let $\omega \leq 0$.
 
For $\xi > -\beta(\omega)\omega$, it holds
\[
\frac{\partial\Lambda}{\partial\xi}(\xi,\omega) = \Lambda(\xi,\omega) \Big(-(1-\alpha)^{-1}(1+\xi)^{-1}-\beta^{-1}(\omega)(1+\beta^{-1}
(\omega)\xi+\omega)^{-1}\Big),
\]
and if $0 < \xi < -\omega\beta(\omega)$, we get
\[
\frac{\partial\Lambda}{\partial\xi}(\xi,\omega)=\Lambda(\xi,\omega) \Big(-(1-\alpha)^{-1}(1+\xi)^{-1}+\beta^{-1}(\omega)(1-\beta^{-1}(\omega)\xi-\omega)^{-1}\Big). 
\]
Now, $-(1-\alpha)^{-1}(1+\xi^\ast)^{-1}+\beta^{-1}(\omega)(1-\beta^{-1}(\omega)\xi^\ast-\omega)^{-1} = 0$ implies 
\[
\xi^\ast=\frac{\beta(\omega)(1-\omega)-1+\alpha}{2-\alpha}.
\]
Since $\frac{\partial\Lambda}{\partial\xi}(\xi,\omega) < 0$ for all $\xi > -\beta(\omega)\omega$ and all $\omega \leq 0$, it follows that for 
$\omega$ fixed, $\Lambda(\ \cdot\ ,\omega)$ 
takes its maximum in one of the points $\xi_1 = 0$, $\xi_2 = -\beta(\omega)\omega$, or, if $\xi^\ast \in (0,-\beta(\omega)\omega)$, 
in $\xi_3 = \xi^\ast$.

It holds $\Lambda(\xi_1,\omega) = 1$ for all $\omega \in \R$, and 
\begin{align*}
\Lambda(\xi_2,\omega) & = (1+(1+|\omega|)^{-\alpha}|\omega|)^{-1/(1-\alpha)}(1+|\omega|) \\
& = (1+|\omega|)^{\alpha/(1-\alpha)}((1+|\omega|)^\alpha+|\omega|)^{-1/(1-\alpha)}(1+|\omega|) \\
& \leq(1+|\omega|)^{-(1-\alpha)/(1-\alpha)+1} \\
& = 1.
\end{align*}
So it remains to check $\xi_3$. We find
\begin{align*}
\Lambda(\xi_3,\omega) & \leq (1+|\xi_3|)^{-1/(1-\alpha)}(1+|\omega|) \\
& = \Big(\frac{2-\alpha+\beta(\omega)(1-\omega)-1+\alpha}{2-\alpha}\Big)^{-1/(1-\alpha)}(1+|\omega|) \\
& \leq2^{1/(1-\alpha)}(1+\beta(\omega)(1-\omega))^{-1/(1-\alpha)}(1+|\omega|) \\
& \leq2^{1/(1-\alpha)}(1+\beta(\omega)|\omega|)^{-1/(1-\alpha)}(1+|\omega|) \\
& = 2^{1/(1-\alpha)}\Lambda(\xi_2,\omega) \\
& \leq 2^{1/(1-\alpha)}.
\end{align*}
This concludes the proof.
\end{proof}

\begin{remark}
Note that for the case $\alpha = 0$, this result is a simple consequence of the submultiplicativity of polynomial weight
\[
(1+|\omega|) \leq (1+|\xi|)(1+|\xi+\omega|).
\]
For $\alpha > 0$ we get in some sense a ``twisted`` submultiplicativity
\[
(1+|\omega|) \leq C (1+|\xi|)^{1/(1-\alpha)}(1+|\beta(\omega)^{-1}\xi+\omega|).
\]
\end{remark}

\begin{lemma}\label{shearlemma}
For $r > 1$ and $\theta > 0$, the following estimate holds:
\[
\int_\mathbb{R}(1+|t|)^{-r}(1+\theta|x-t|)^{-r}dt\leq C\big(\theta^{-1}(1+|x|)^{-r}+(1+\theta|x|)^{-r}\big).
\]
\end{lemma}
(For the proof, see \cite[Lemma 3.1]{dastte11}.)

We are now able to complete proof of Theorem \ref{kern-integr-th}.

\textit{Proof of Theorem \ref{kern-integr-th}.}
First, the conditions imposed on $r$ imply that $r > \max\big\{1,\frac{\alpha}{2(1-\alpha)}\big\}$, i.e. $\psi,\varphi$ 
are admissible and Lemma \ref{symbol-derivatives-estimate} is applicable.

We rewrite the kernel $K_{\psi,\varphi}^\kappa$ as follows:
\begin{align*}
K_{\psi,\varphi}^\kappa&(x,\omega, x^\ast,\omega^\ast)\\ & = \langle  M_{-x}T_\omega D_{1/\beta(\omega)}\hat\psi,
m^{-\kappa}_\psi M_{-x^\ast}T_{\omega^\ast} D_{1/\beta(\omega^\ast)}\hat\varphi\rangle \\
& = \langle  M_{-x}T_\omega D_{1/\beta(\omega)}\hat\psi,
M_{-x^\ast} T_{\omega^\ast} D_{1/\beta(\omega^\ast)} (D_{\beta(\omega^\ast)}T_{-\omega^\ast}m^{-\kappa}_\psi)\hat\varphi\rangle \\
& = e^{2\pi i \omega^\ast(x^\ast-x)}\ \langle D_{\beta(\omega^\ast)}T_{-\omega^\ast} M_{-(x-x^\ast)}T_\omega D_{1/\beta(\omega)}\hat\psi,
h_{\omega^\ast,\kappa}\hat\varphi\rangle \\
& = e^{2\pi i \omega^\ast(x^\ast-x)}\ \langle M_{-(x-x^\ast)/\beta(\omega^\ast)}T_{\beta(\omega^\ast)(\omega-\omega^\ast)} D_{\beta(\omega^\ast)/\beta(\omega)}\hat\psi,
h_{\omega^\ast,\kappa}\hat\varphi\rangle.
\end{align*}
If we plug this in \eqref{finite-integral-kernels} and use the substitutions $(x-x^\ast)/\beta(\omega^\ast)\mapsto x$, 
$\beta(\omega^\ast)(\omega-\omega^\ast)\mapsto \omega$, the 
notation $\widetilde w_s$ defined in Lemma \ref{weightlemma} and
\[
 \theta(\omega,\omega^\ast) := \frac{\beta(\omega^\ast+\beta^{-1}(\omega^\ast)\omega)}{\beta(\omega^\ast)}
  = \left(\dfrac{1+|\omega^\ast|}{1+|\omega^\ast+(1+|\omega^\ast|)^\alpha\omega|}\right)^\alpha,
\]
we get
 \begin{align*}
&\int_{\R} \int_{\R}  \left|K_{\psi,\varphi}^\kappa(x,\omega,x^\ast,\omega^\ast)\right|
w_s(\omega,\omega^\ast)\;dxd\omega \\
& = \int_{\R}\int_{\R}|\langle M_{-(x-x^\ast)/\beta(\omega^\ast)}T_{\beta(\omega^\ast)(\omega-\omega^\ast)} 
D_{\beta(\omega^\ast)/\beta(\omega)}\hat\psi,h_{\omega^\ast,\kappa}\hat\varphi\rangle|w_s(\omega,\omega^\ast)\;dxd\omega \\
& = \int_{\R}\int_{\R}|\langle M_{-x}T_{\omega} 
D_{\theta^{-1}(\omega,\omega^\ast)}\hat\psi,h_{\omega^\ast,\kappa}\hat\varphi\rangle|\widetilde w_s(\omega,\omega^\ast)\;dx d\omega \\
& = \int_{\R}\int_{\R}|\mathcal{F}((T_{\omega} 
D_{\theta^{-1}(\omega,\omega^\ast)})(\hat\psi) h_{\omega^\ast,\kappa}\overline{\hat\varphi})(x)|\widetilde w_s(\omega,\omega^\ast)\;dx d\omega.
\end{align*}
Define
\[
G_{\omega,\omega^\ast}(\xi) := T_\omega D_{\theta(\omega,\omega^\ast)^{-1}}\hat\psi(\xi)\cdot 
\overline{\hat\varphi}(\xi)\cdot h_{\omega^\ast,\kappa}(\xi).
\]
 $G_{\omega,\omega^\ast}$ is twice continuously differentiable as $\hat\psi,\hat\varphi\in C^2(\R)$ by assumption and
 $h_{\omega^\ast,\kappa}\in C^2(\R)$ by Corollary \ref{estimate-symbol-term}.
Hence, as already explained in \eqref{pointwise-bound-G}, $\mathcal{F}(G_{\omega,\omega^\ast})$ can be estimated pointwise by
$$
 |\mathcal{F}(G_{\omega,\omega^\ast})(x)| \leq C \max\left\{\|G_{\omega,\omega^\ast}\|_1,
\|G_{\omega,\omega^\ast}^{(2)}\|_1\right\}\min\left\{1,\frac{1}{ x^2}\right\}.
$$
For $n\in\{0,2\}$ and $n_1,n_2,n_3\in\{0,1,2\}$ one gets
\[
 |G_{\omega,\omega^\ast}^{(n)}(\xi)|
 \leq C \sum \Big|\theta(\omega,\omega^\ast)^{n_1}T_\omega D_{\theta(\omega,\omega^\ast)^{-1}}\hat\psi^{(n_1)}(\xi)
 \hat\varphi^{(n_2)}(\xi)h_{\omega^\ast,\kappa}^{(n_3)}(\xi)\Big|,
\]
where the sum is taken over all triples $\{ n_1, n_2, n_3\}$ with $n_1 + n_2 + n_3 = n$.
In order to estimate this term we split it and observe that $\theta(\omega,\omega^\ast)\leq\widetilde m_\alpha(\omega,\omega^\ast)$. Thus, 
Lemma \ref{weightlemma}  ($s=\alpha$) yields
\begin{equation}\label{eq-p27}
\theta(\omega,\omega^\ast)\leq (1+|\omega|)^{\alpha/(1-\alpha)}.
\end{equation}
Hence
\begin{align*}
 |\theta & (\omega,\omega^\ast)^{n_1}T_\omega D_{\theta(\omega,\omega^\ast)^{-1}}\hat\psi^{(n_1)}(\xi)| \\
& \leq C (1+|\omega|)^{n_1\alpha/(1-\alpha)}\theta(\omega,\omega^\ast)^{1/2}(1+\theta(\omega,\omega^\ast)|\xi-\omega|)^{-r} \\
& \leq C (1+|\omega|)^{2\alpha/(1-\alpha)}\theta(\omega,\omega^\ast)^{1/2}
 (1+\theta(\omega,\omega^\ast)|\xi-\omega|)^{-r+2\alpha/(1-\alpha)}.
\end{align*}
The second part may be estimated using Corollary \ref{estimate-symbol-term} and Lemma \ref{Lambda-bounded}; this yields
\begin{align*}
| \hat\varphi^{(n_2)} & (\xi)h_{\omega^\ast,\kappa}^{(n_3)}(\xi)| \\
 \leq\ & C (1+|\xi|)^{-r}
\Big( (1+|\beta^{-1}(\omega^\ast)\xi+\omega^\ast|)^{-1}(1+|\omega^\ast|)\Big)^{n_3\alpha} \\
 =\ &C (1+|\xi|)^{-r+n_3\alpha/(1-\alpha)} \\
& \cdot\Big((1+|\xi|)^{-1/(1-\alpha)} (1+|\beta^{-1}(\omega^\ast)\xi+\omega^\ast|)^{-1}(1+|\omega^\ast|)\Big)^{n_3\alpha} \\
 \leq\ & C (1+|\xi|)^{-r+2\alpha/(1-\alpha)}
\Lambda(\xi,\omega^\ast)^{n_3\alpha}\leq C (1+|\xi|)^{-r+2\alpha/(1-\alpha)}.
\end{align*}
Consequently,  summarizing our previous considerations, we get
\begin{align*}
&|G_{\omega,\omega^\ast}^{(n)}(\xi)|\\ &\leq C(1+|\omega|)^{2\alpha/(1-\alpha)}\theta(\omega,\omega^\ast)^{1/2}
\Big[\big(1+\theta(\omega,\omega^\ast)|\xi-\omega|\big)\big(1+|\xi|\big)\Big]^{-r+2\alpha/(1-\alpha)}.
\end{align*}
Now, it is possible to apply  Lemma \ref{shearlemma}, as $r-2\alpha/(1-\alpha)>1$. Together with
$\theta^{-1}(\omega,\omega^\ast)\leq (1+|\omega|)^\alpha$
(and thus $\theta(\omega,\omega^\ast)\geq (1+|\omega|)^{-\alpha}$) and \eqref{eq-p27}, it follows
\begin{align*}
\big\|  G_{\omega,\omega^\ast}^{(n)}\big\|_1 
& \leq C (1+|\omega|)^{2\alpha/(1-\alpha)}\Big[\theta^{-1/2}(\omega,\omega^\ast)(1+|\omega|)^{-r+2\alpha/(1-\alpha)} \\
& \quad \quad +\theta^{1/2}(\omega,\omega^\ast)
(1+\theta(\omega,\omega^\ast)|\omega|)^{-r+2\alpha/(1-\alpha)}\Big] \\
& \leq C (1+|\omega|)^{2\alpha/(1-\alpha)}\Big[(1+|\omega|)^{\alpha/2-r+2\alpha/(1-\alpha)}\\
& \quad \quad + (1+|\omega|)^{\alpha/2(1-\alpha)}\big(1+(1+|\omega|)^{-\alpha}|\omega|\big)^{-r+2\alpha/(1-\alpha)}\Big] \\
& \leq C (1+|\omega|)^{2\alpha/(1-\alpha)}\Big[(1+|\omega|)^{\alpha/2-r+2\alpha/(1-\alpha)}\\
 & \quad \quad + (1+|\omega|)^{\alpha/2(1-\alpha)-(1-\alpha)r+2\alpha}\Big] \\
& \leq C (1+|\omega|)^{2\alpha/(1-\alpha)+\alpha/2(1-\alpha)-(1-\alpha)r+2\alpha}\\
& =C(1+|\omega|)^{(9\alpha-4\alpha^2)/2(1-\alpha)-(1-\alpha)r},
\end{align*}
where we have used that $$\big(1+(1+|\omega|)^{-\alpha}|\omega|\big)\geq (1+|\omega|)^{-\alpha}(1+|\omega|)$$ and that the condition on $r$
ensures that $$\alpha/2-r+2\alpha/(1-\alpha)<\alpha/2(1-\alpha)-(1-\alpha)r+2\alpha.$$
 Hence, by Lemma \ref{weightlemma}
\begin{align*}
\int_\R\max & \left\{\|G_{\omega,\omega^\ast}\|_1,\|G_{\omega,\omega^\ast}^{(2)}\|_1\right\} \widetilde{w}_s(\omega,\omega^\ast)\;d\omega \\
& \leq C\int_\R (1+|\omega|)^{(9\alpha-4\alpha^2+2s)/2(1-\alpha)-(1-\alpha)r}\;d\omega,
\end{align*}
which is finite whenever $(1-\alpha)r - \frac{9\alpha-4\alpha^2+2s}{2(1-\alpha)} > 1$, or equivalently
\[
r > \frac{2+2s+7\alpha-4\alpha^2}{2(1-\alpha)^2}.
\]
Finally,
\begin{align*}
&\int_{\R}\int_{\R} \left|K_{\psi,\varphi}^\kappa(x,\omega,x^\ast,\omega^\ast)\right|
w_s(\omega,\omega^\ast)\;dxd\omega \\
& \leq C \int_{\R}\max\left\{\|G_{\omega,\omega^\ast}\|_1,
\|G_{\omega,\omega^\ast}^{(2)}\|_1 \right\}\widetilde w_s(\omega,\omega^\ast)\;d\omega
\int_\R \min\left\{1,\frac{1}{ x^2}\right\}\;dx \\
& = C \int_{\mathbb{R}}\min\left\{1,\frac{1}{ x^2}\right\}\;dx \\
& < \infty.
\end{align*}
\hfill $\Box$

\begin{remark}
Assume we also impose a polynomial weight in $x$, i.e. 
\begin{align*}
 w_{s,t} & (x,x^\ast,\omega,\omega^\ast) \\
& := \max\left\{\left(\frac{1+|\omega|}{1+|\omega^\ast|}\right)^{s}\left(\frac{1+|x|}{1+|x^\ast|}\right)^{t},
 \left(\frac{1+|\omega^\ast|}{1+|\omega|}\right)^{s}\left(\frac{1+|x^\ast|}{1+|x|}\right)^{t}\right\} \\
& \leq \max\left\{\frac{1+|\omega|}{1+|\omega^\ast|},
\frac{1+|\omega^\ast|}{1+|\omega|}\right\}^{s}\max\left\{\frac{1+|x|}{1+|x^\ast|},
\frac{1+|x^\ast|}{1+|x|}\right\}^{t}.
\end{align*}
with $s,t>0$.
Easy calculations then show that 
\[
\sup_{z\in\mathbb{R}}\max\left\{\frac{1+|z|}{1+|x+z|},\frac{1+|x+z|}{1+|z|}\right\}=1+|x| 
\]
Hence by Lemma \ref{weightlemma}
\begin{align*}
 \widetilde w_s(x,\omega,x^\ast,\omega^\ast) & := w_s(x^\ast+\beta(\omega^\ast)x,x^\ast,\omega^\ast+\beta(\omega^\ast)^{-1}\omega,\omega^\ast) \\
& \leq (1+|\omega|)^{\frac{s}{1-\alpha}}(1+\beta(\omega^\ast)|x|)^{t} \\
&\leq (1+|\omega|)^{\frac{s}{1-\alpha}}(1+|x|)^{t}.
 \end{align*}
Hence, we need the derivatives of degree up to $k=\lfloor t\rfloor +2$
and get
\[
|\mathcal{F}(G_{\omega,\omega^\ast})(x)|\leq \min\left\{\|G_{\omega,\omega^\ast}\|_1,\frac{\|G_{\omega,\omega^\ast}^{(k)}\|_1}{(2\pi|x|)^k}
\right\}.
\]
One may then proceed as before.
\end{remark}


\subsection{The admissible covering}\label{subsec:adm_cov}


We define the family $\mathcal{U}^\varepsilon$ of open subsets in $\R\times\R$ by: 
$\mathcal{U}^\varepsilon := \{U_{j,k}^\varepsilon\}_{j,k\in \mathbb{Z}}$, with
\[
U_{j,k}^\varepsilon := \varepsilon\beta(\omega_j)(k-1,k+1)\times(\omega_j-2\varepsilon c\beta(\omega_j)^{-1},
\omega_j+2\varepsilon c\beta(\omega_j)^{-1}),
\]
\[
\omega_j := p_\alpha(\varepsilon j),
\]
and
\[
p_\alpha(\omega) := sgn(\omega)\Big(\big(1+(1-\alpha)|\omega|\big)^{1/(1-\alpha)}-1\Big).
\]

Similar arguments as in the proof of \cite[Theorem 5.12.]{DahForRau+2008} (we intend to cover $X$ and not $G_{aWH}$) show 
that $\mathcal{U}^\varepsilon$ is an admissible covering for $X\simeq\R\times\R$.
Hence, we omit the proof here. As the area $|U_{j,k}^\varepsilon| = 2\varepsilon\beta(\omega_j)\cdot 4\varepsilon c \beta(\omega_j)^{-1}
 = 8\varepsilon^2c$ is constant, it follows that $\mathcal{U}^\varepsilon$ is a moderate admissible covering for $X$.
Moreover, \cite[Lemma 5.14]{DahForRau+2008} guarantees that
\[
\sup_{j,k\in\Z}\ \ \sup_{(x,\omega),(x^\ast,\omega^\ast)\in U_{j,k}^\varepsilon}m_s(x,\omega,x^\ast,\omega^\ast) \leq C_{m_s,\mathcal{U}^\varepsilon} \leq C_{m_s}< \infty.
\]
with $C_{m_s}$ independent of $\varepsilon<\varepsilon_0$.

\subsection{The integrability of the local oscillations kernel}\label{subsec:osc_integr}


\begin{theorem}\label{osc-integr-th}
Let $\mathcal{U}^\varepsilon$ be the admissible covering for the $\alpha$-modulation transform defined above, $osc_{\mathcal{U}^\varepsilon,\Gamma}$
be the local oscillations
kernel
as defined in \eqref{loc-osc-defined} with
$\Gamma:X\times X\rightarrow\C$, $|\Gamma|\equiv 1$, appropriately chosen, 
$s \geq 0$ and $\psi\in L^2(\R)$, such that 
$\hat\psi\in C^3(\R)$ fulfills  $|\hat\psi^{(n)}(\xi)| \leq C (1+|\xi|)^{-r}$, for $n=0,1,2,3$, with parameter $r$ 
satisfying
\begin{equation}\label{decay-constant2}
 r>\frac{2+2s+7\alpha-4\alpha^2}{2(1-\alpha)^2}+1.
\end{equation}
Then, using the notation of \eqref{gamma1} and \eqref{gamma2}, it holds 
$$\gamma(\varepsilon)=\max\{\gamma_1(\varepsilon),\gamma_2(\varepsilon)\}\rightarrow 0, \mbox{ for } \varepsilon\rightarrow 0,$$
where 
\[
 \gamma_1(\varepsilon)= \sup\limits_{(x^\ast,\omega^\ast) \in \R^2}\int_{\R}\int_{\R}\left|osc_{\mathcal{U}^\varepsilon,\Gamma}
 (x,\omega,x^\ast,\omega^\ast)\right|w_s(\omega,\omega^\ast)\;dxd\omega 
\]
and 
\[
  \gamma_2(\varepsilon)=\sup\limits_{(x,\omega) \in \R^2}\int_{\R}\int_{\R}\left|osc_{\mathcal{U}^\varepsilon,\Gamma}
 (x,\omega,x^\ast,\omega^\ast)\right|w_s(\omega,\omega^\ast)\;dx^\ast d\omega^\ast .
\]
\end{theorem}
\begin{remark}
 Theorem \ref{osc-integr-th} together with the considerations of Section \ref{subsec:adm_cov} therefore show that the
 assumptions of Theorem \ref{discretization_result} are met, i.e. discretization is possible.
\end{remark}
We need the following Lemma.
\begin{lemma}\label{lemma:constants-covering}
Let  $\mathcal{U}^\varepsilon$ be the admissible covering defined in Section \ref{subsec:adm_cov}. There exist  two constants $C_1,C_2>0$,
independent of $x,\omega\in\R \mbox{ and } 0<\varepsilon<\varepsilon_0$, such that
\begin{equation*}\label{eq-Q-x-w}
 (x,\omega) - Q_{x,\omega} \subseteq \big(- C_1 \varepsilon \beta(\omega), C_1 \varepsilon \beta(\omega)\big)\times
 \big(- C_2 \varepsilon \beta(\omega)^{-1}, C_2 \varepsilon \beta(\omega)^{-1}\big).
\end{equation*}
 \end{lemma}
 
 \begin{proof}
 Let $N$ be the constant from the
 finite overlap property (Definition \ref{def-admissible-covering} (ii)). Observe that,
\[
  p_\alpha'(\omega)=\big(1+(1-\alpha)|\omega|\big)^{\alpha/(1-\alpha)}=\beta(p_\alpha(\omega))^{-1}.
\]
As the sampling points are symmetrically distributed we may assume w.l.o.g. that $\omega \geq 0$.
Let $(y,\eta)\in Q_{x,\omega}$ and $j^\ast>0$ be the smallest index such that $\omega_{j^\ast} >\omega$. Then, 
$0\leq \omega_{j^\ast-1}\leq\omega<\omega_{j^\ast}$. Using the Mean Value Theorem, for  
$\delta\in(\varepsilon(j^\ast-N),\varepsilon(j^\ast+N))$, and submultiplicativity yields
\begin{align*}
|\eta-\omega| & \leq |\omega_{j^\ast+N}+2\varepsilon c\beta(\omega_{j^\ast+N})^{-1}-(\omega_{j^\ast-N}-2\varepsilon c\beta(\omega_{j^\ast-N})^{-1})| \\
& \leq |\omega_{j^\ast+N}-\omega_{j^\ast-N}|+4\varepsilon c \beta(\omega_{j^\ast+N})^{-1} \\
& = |p_\alpha(\varepsilon(j^\ast+N)) - p_\alpha(\varepsilon(j^\ast-N))| + 4 \varepsilon c \beta(\omega_{j^\ast+N})^{-1} \\
& =2N \varepsilon\beta(p_\alpha(\delta))^{-1} + 4 \varepsilon c \beta(\omega_{j^\ast+N})^{-1}\\
& \leq C \varepsilon \beta(\omega_{j^\ast+N})^{-1} \\
& = C \varepsilon \big(1+(1-\alpha)\varepsilon(j^\ast-1+N+1)\big)^{\alpha/(1-\alpha)} \\
& \leq C \varepsilon \big(1+(1-\alpha)\varepsilon(j^\ast-1)\big)^{\alpha/(1-\alpha)}\big(1+(1-\alpha)\varepsilon(N+1)\big)^{\alpha/(1-\alpha)} 
\\
& = C_2 \varepsilon\beta(\omega_{j^\ast-1})^{-1} \\
& \leq  C_2 \varepsilon\beta(\omega)^{-1}.
\end{align*}

Every point $x\in\R$ is contained in at most two of the intervals 
$\Omega_{j,k}^\varepsilon := \varepsilon\beta(\omega_j)(k-1,k+1)$ if $j\in \mathbb{Z}$  is fixed. Hence, 
$(y,\eta)\in U_{j,k}^\varepsilon\subset Q_{x,\omega}$ implies 
that $|y-x| \leq 2 \varepsilon\beta(\omega_j)$. 
As $(x_{j,k},\omega_j)\in Q_{x,\omega}$ it follows by previous calculations that
\[
|\omega-\omega_j|\leq C_2\varepsilon\beta(\omega)^{-1}.
\]
Assuming that $\omega\geq 1$ and $C_2\varepsilon<1/2$ yields
\[
|\omega_j|\geq |\omega|-C_2\varepsilon\beta(\omega)^{-1}\geq 0.
\]
As $\beta$ is monotonically decreasing we therefore get 
\begin{align*}
|y-x| & \leq 2 \varepsilon\beta(\omega_{j^\ast}) \\
& \leq 2 \varepsilon\beta\big(|\omega|-C_2\varepsilon\beta(\omega)^{-1}\big) \\
& = 2 \varepsilon\big(1+|\omega|-C_2\varepsilon(1+|\omega|)^\alpha\big)^{-\alpha} \\
& = 2 \varepsilon(1+|\omega|)^{-\alpha}\big(1-C_2\varepsilon(1+|\omega|)^{\alpha-1}\big)^{-\alpha} \\
& = 2 \varepsilon\beta(\omega)(1-C_2\varepsilon)^{-\alpha} \\
& \leq 2^{1+\alpha}\varepsilon\beta(\omega).
\end{align*}
If $0\leq\omega<1$ we have
\begin{align*}
|y-x| & \leq 2 \varepsilon\beta(\omega_{j^\ast}) \\
& \leq 2 \varepsilon\beta(0) \\
& = 2 \varepsilon\beta(1)^{-1}\beta(1) \\
& \leq 2^{1+\alpha} \varepsilon\beta(\omega),
\end{align*}
which concludes the proof.
\end{proof}

\textit{Proof of Theorem \ref{osc-integr-th}.}
First, we rewrite 
\begin{align*}
&\mathcal{R}(x,\omega,x^\ast,\omega^\ast)-
\Gamma(x^\ast,\omega^\ast,y,\eta)\mathcal{R}(x,\omega,y,\eta) \\
& = 
  \langle A_\sigma^{-1}\pi(\sigma(x,\omega))\psi,\pi(\sigma(x^\ast,\omega^\ast))\psi
  -\overline{\Gamma(x^\ast,\omega^\ast,y,\eta)}
  \pi(\sigma(y,\eta))\psi\rangle \\
& =
  \langle A^{-1}_\sigma T_xM_\omega D_{\beta(\omega)}\psi,T_{x^\ast}M_{\omega^\ast} D_{\beta(\omega^\ast)}
  \psi-\overline{\Gamma(x^\ast,\omega^\ast,y,\eta)}
  T_{y}M_{\eta} D_{\beta(\eta)}\psi\rangle \\
& = 
  \langle m^{-1}M_{-x}T_\omega D_{1/\beta(\omega)}\hat\psi,M_{-x^\ast}T_{\omega^\ast} D_{1/\beta(\omega^\ast)}
\Upsilon(x^\ast,\omega^\ast,y,\eta)\hat\psi\rangle,
\end{align*}
where the operator $\Upsilon$ is defined by
\begin{align*}
& \Upsilon (x^\ast,\omega^\ast,y,\eta) \\
& := I-\overline{\Gamma(x^\ast,\omega^\ast,y,\eta)}e^{2\pi i\omega^\ast(x^\ast-y)}
  M_{-(y-x^\ast)/\beta(\omega^\ast)}T_{(\eta-\omega^\ast)\beta(\omega^\ast)} D_{\beta(\omega^\ast)/\beta(\eta)}.
\end{align*}
 
Setting $\Gamma(x^\ast,\omega^\ast,y,\eta) := e^{-2\pi i\omega^\ast(x^\ast-y)}$ yields
\[
 \Upsilon(x^\ast,\omega^\ast,y,\eta)
  = I-M_{-(y-x^\ast)/\beta(\omega^\ast)}T_{(\eta-\omega^\ast)\beta(\omega^\ast)} D_{\beta(\omega^\ast)/\beta(\eta)}.
\]

 By Lemma \ref{lemma:constants-covering},
there exist constants $C_1,C_2$ independent of $x^\ast,\omega^\ast\in\R$ and $\varepsilon>0$ such that
\[
 (x^\ast,\omega^\ast) - Q_{x^\ast,\omega^\ast} \subseteq (- C_1 \varepsilon \beta(\omega^\ast), C_1 \varepsilon \beta(\omega^\ast))\times
 (- C_2 \varepsilon \beta(\omega^\ast)^{-1}, C_2 \varepsilon \beta(\omega^\ast)^{-1})
\]
Hence, $|y-x^\ast|/\beta(\omega^\ast)\leq C_1 \varepsilon$ and $|\eta-\omega^\ast|\beta(\omega^\ast)\leq C_2 \varepsilon$
for all $(y,\eta)\in Q_{x^\ast,\omega^\ast}$. Moreover, 
\begin{equation*}
  \frac{\beta(\omega^\ast)}{\beta(\eta)} = \left(\frac{1+|\eta|}{1 + |\omega^\ast|}\right)^\alpha
  \leq \left(\frac{1+|\omega^\ast| + C_2\varepsilon\beta(\omega^\ast)^{-1}}{1 + |\omega^\ast|}\right)^\alpha
 \end{equation*}
 \begin{equation*}
  = \left(1+\frac{C_2 \varepsilon(1 + |\omega^\ast|)^\alpha}{1 + |\omega^\ast|}\right)^\alpha \leq (1 + C_2 \varepsilon)^\alpha \leq 1 + C_2 \varepsilon
 \end{equation*}
 and for $C_2\varepsilon < 1$,
 \begin{equation*}
  \frac{\beta(\omega^\ast)}{\beta(\eta)}
  \geq \left(\frac{1 + |\omega^\ast| - C_2 \varepsilon\beta(\omega^\ast)^{-1}}{1 + |\omega^\ast|}\right)^\alpha
 \end{equation*}
 \begin{equation*}
  = \left(1 - \frac{C_2 \varepsilon(1 + |\omega^\ast|)^\alpha}{1 + |\omega^\ast|}\right)^\alpha \geq(1-C_2 \varepsilon)^\alpha \geq 1 - C_2  \varepsilon.
 \end{equation*}
Consequently, for $C_2\varepsilon < 1/2$, it holds
 \[
  1 - 2 C_2 \varepsilon \leq \frac{1}{1 + C_2 \varepsilon} \leq \frac{\beta(\eta)}{\beta(\omega^\ast)} \leq
  \frac{1}{1 - C_2 \varepsilon} \leq  1 + 2 C_2 \varepsilon.
 \]

 Hence, the generalized oscillation kernel can be estimated as follows
\begin{align*}
&osc_{\mathcal{U} ,\Gamma}(x,\omega,x^\ast,\omega^\ast) \\
& = \sup_{(y,\eta)\in Q_{x^\ast,\omega^\ast}}
  |\langle m^{-1}M_{-x}T_\omega D_{1/\beta(\omega)}\hat\psi,M_{-x^\ast}T_{\omega^\ast} D_{1/\beta(\omega^\ast)}
\Upsilon(x^\ast,\omega^\ast,y,\eta)\hat\psi\rangle|\\
&\leq \sup_{\lambda,\mu,\nu\in \sigma(-\varepsilon,\varepsilon)}
  |\langle m^{-1}M_{-x}T_\omega D_{1/\beta(\omega)}\hat\psi,M_{-x^\ast}T_{\omega^\ast} D_{1/\beta(\omega^\ast)}
\Phi(\lambda,\mu,\nu)\hat\psi\rangle|, 
\end{align*}
 where $\sigma := \max\{C_1, 2 C_2\}$ is independent of $x^\ast,\omega^\ast$ and
 $$
 \Phi(\lambda,\mu,\nu):=I-M_{\mu}T_{\lambda} D_{1+\nu}.
 $$
 The main idea of the proof is to show that
\begin{equation*}
 \Big|\frac{d^{k}}{d\xi^k}\Phi(\lambda,\mu,\nu) (\hat\psi)(\xi)\Big|\leq C\delta(\varepsilon)(1+|\xi|)^{-t},\ k=0,1,2,
\end{equation*}
with $\delta(\varepsilon)\rightarrow 0$, if $\varepsilon\rightarrow 0$, and  $t$ satisfies
 \eqref{decay-constant1}. Then, following the arguments of Theorem \ref{kern-integr-th} yields the stated result, i.e.
 $\gamma(\varepsilon)=C\delta(\varepsilon)$ which converges to zero.

Short calculations show that $|M_{\mu}T_{\lambda}D_{(1+\nu)^{-1}}\hat\psi(\xi)|\leq C(1+|\xi|)^{-r}$, if $\nu,\lambda\leq C$ and 
consequently $$|2\pi i\xi\mu M_{\mu}T_{\lambda}D_{(1+\nu)^{-1}}\hat\psi(\xi)|\leq C|\mu|(1+|\xi|)^{-r+1}.$$
For all $\lambda,\mu,\nu\in \sigma(-\varepsilon,\varepsilon)$, it holds
 \begin{align*}
 \Big|\frac{d}{d\xi}\Phi&(\lambda,\mu,\nu) (\hat\psi)(\xi)\Big|\\
  &=\big|\hat\psi'(\xi)-M_{\mu}T_{\lambda}
  D_{(1+\nu)^{-1}}\big(2\pi i\mu \hat\psi(\xi)+(1+\nu)\hat\psi'(\xi)\big)\big| \\
  &\leq C\left\{
  \big|\big(I-(1+\nu)M_{\mu}T_{\lambda}D_{(1+\nu)^{-1}}\big)\hat\psi'(\xi)\big|+|\mu|(1+|\xi|)^{-r}\right\}\\
 &\leq C
 \left\{ \big|\big(I-(1+\nu)M_{\mu}T_{\lambda}D_{(1+\nu)^{-1}}\big)\hat\psi'(\xi)\big|+\varepsilon(1+|\xi|)^{-r}\right\}.
 \end{align*}
Repeating the previous argument yields, for $k=0,1,2$,
$$
 \Big|\frac{d^{k}}{d\xi^k}\Phi(\lambda,\mu,\nu) (\hat\psi)(\xi)\Big| $$ $$
  \leq C
  \left\{\big|\big(I-(1+\nu)^kM_{\mu}T_{\lambda}D_{(1+\nu)^{-1}}\big)\hat\psi^{(k)}(\xi)\big|+\varepsilon (1+|\xi|)^{-r}\right\}
$$
 It remains to estimate the first term.  By the triangle inequality one has
   \begin{align*}
& |\hat\psi^{(k)}(\xi)-(1+\nu)^{k}M_{\mu}T_{\lambda}D_{(1+\nu)^{-1}}\hat\psi^{(k)}(\xi)| \\
& \leq C\Big\{ |1-(1+\nu)^{k+1/2}||\hat\psi^{(k)}(\xi)|+ \underbrace{ |1-e^{2\pi i\mu\xi}||\hat\psi^{(k)}(\xi)|}_{(I)}+ \\
& \quad +   \underbrace{ |\hat\psi^{(k)}(\xi)-\hat\psi^{(k)}(\xi-\lambda)| }_{(II)}+  
\underbrace{ |\hat\psi^{(k)}(\xi-\lambda)-\hat\psi^{(k)}\big((1+\nu)(\xi-\lambda)\big)|}_{(III)}\Big\}.
\end{align*}

\textit{Ad (I):} Recall that,  for $0 < |x| < 1/2$, $|1-e^{2\pi i x}| = |1-e^{2\pi i |x|}|$ is strictly increasing in $|x|$.
We split the real line into the parts $|\xi| < |\mu|^{-1/2}$ and $|\xi| \geq |\mu|^{-1/2}$.

Let $|\xi| < |\mu|^{-1/2}$ and $|\mu|^{1/2}\leq (\sigma\varepsilon)^{1/2}<1/2$, then
\begin{align*}
  |1-e^{2\pi i\mu\xi}||\hat\psi^{(k)}(\xi)|&\leq C\big|1-e^{2\pi i|\mu|^{1/2}}\big|(1+|\xi|)^{-r}\\
 & \leq  C\big|1-e^{2\pi i|\mu|^{1/2}}\big|(1+|\xi|)^{-r+1}.
\end{align*}
If $|\xi| \geq |\mu|^{-1/2}$, it holds
\begin{align*}
|1-e^{2\pi i\mu\xi}||\hat\psi^{(k)}(\xi)| &\leq 2|\hat\psi^{(k)}(\xi)| \leq 2 C (1+|\xi|)^{-r} \\
& \leq 2 C \big(1+|\mu|^{-1/2}\big)^{-1}(1+|\xi|)^{-r+1}.
\end{align*}
Hence, $\forall \xi\in\R$ 
\begin{align*}
|1-e^{2\pi i\mu\xi}|&|\hat\psi^{(k)}(\xi)| \\ & \leq C\max\Big\{\big|1-e^{2\pi i|\mu|^{1/2}}\big|,\big(1+|\mu|^{-1/2}\big)^{-1}\Big\}(1+|\xi|)^{-r+1}
\\ & \leq C\max\Big\{\big|1-e^{2\pi i(\sigma\varepsilon)^{1/2}}\big|,\big(1+(\sigma\varepsilon)^{-1/2}\big)^{-1}\Big\}(1+|\xi|)^{-r+1}.
\end{align*}

\textit{Ad (II):}
By the Mean Value Theorem, it holds for some $\xi^\ast\in(0,\lambda)$, with  $|\lambda|\leq\sigma\varepsilon<1$
\begin{align*}
  |\hat\psi^{(k)}(\xi) - \hat\psi^{(k)}(\xi-\lambda)|&= |\lambda||\hat\psi^{(k+1)}(\xi-\xi^\ast)| \leq C |\lambda|(1+|\xi-\xi^\ast|)^{-r}\\
  &\leq C \sigma\varepsilon(1+|\xi|)^{-r}.
\end{align*}
\textit{Ad (III):} 
Applying  the Mean Value Theorem again yields 
\begin{align*}
|\hat\psi^{(k)}(\xi-\lambda) & -\hat\psi^{(k)}((1+\nu)(\xi-\lambda))| \\
& =|\nu(\xi-\lambda)| |\hat\psi^{(k+1)}(\xi-\lambda+\xi^\ast)| \\
& \leq C |\nu|(1+|\xi|)(1+|\xi-\lambda+\xi^\ast|)^{-r} \\
& \leq C |\nu|(1+|\xi|)\big(1+(1-|\nu|)(|\xi|-|\lambda|)\big)^{-r} \\
& \leq C \sigma\varepsilon(1+|\xi|)^{-r+1},
\end{align*}
 for some $\xi^\ast\in(0,\nu(\xi-\lambda))$, with
$|\nu|<\sigma\varepsilon<1$ and $(1-|\nu|)|\lambda| < \sigma\varepsilon < 1$.

Using all previous considerations we obtain for $\mu,\lambda,\nu\in \sigma(-\varepsilon,\varepsilon)$
  \begin{align*}
& |\hat\psi^{(k)}(\xi)-(1+\nu)^{k}M_{\mu}T_{\lambda}D_{(1+\nu)^{-1}}\hat\psi^{(k)}(\xi)| \\
& \leq C \max\left\{|1-(1+\sigma\varepsilon)^{k+1/2}|,\big|1-e^{2\pi i(\sigma\varepsilon)^{1/2}}\big|,
\big(1+(\sigma\varepsilon)^{-1/2}\big)^{-1},\sigma\varepsilon\right\}\\ &\quad\quad\cdot(1+|\xi|)^{-r+1}.
\end{align*}
All in all, we have shown that
\begin{equation*}
 \Big|\frac{d^{k}}{d\xi^k}\Phi(\lambda,\mu,\nu) (\hat\psi)(\xi)\Big|\leq C\delta(\varepsilon)(1+|\xi|)^{-r+1},\ k=0,1,2,
\end{equation*}
which concludes the proof as $t:=r-1$ satisfies \eqref{decay-constant1} by \eqref{decay-constant2}.
\hfill $\Box$


\section*{Acknowledgement}


This work was funded by the Austrian Science Fund (FWF) DACH-project BIOTOP('Adaptive Wavelet and Frame techniques 
for acoustic BEM';I-1018-N25), by the FWF START-project FLAME ('Frames and Linear Operators for Acoustical 
Modeling and Parameter Estimation'; Y 551-N13) and  the DFG Project Number DA 360/19-1. We would like to thank 
all the project members for valuable discussions and comments.


\bibliographystyle{plain}
\bibliography{biotop_coorbit}

\section*{Affiliations}

Michael Speckbacher, Peter Balazs\\
Acoustics Research Institute\\
Austrian Academy of Sciences\\
Wohllebengasse 12-14\\
1040 Wien\\
Austria\\
E-mail: \email{speckbacher@kfs.oeaw.ac.at}, peter.balazs@oeaw.ac.at\\

Dominik Bayer\\
Universit{\"a}t der Bundeswehr M{\"u}nchen\\
Inst. f{\"u}r Mathematik und Informatik\\
Werner-Heisenberg-Weg 39\\
85577 Neubiberg\\
Germany\\
E-mail: dominik.bayer@unibw.de\\

Stephan Dahlke\\
Philipps-Universit{\"a}t Marburg\\
FB12 Mathematik und Informatik\\
Hans-Meerwein Stra{\ss}e\\
Lahneberge\\
35032 Marburg\\
E-mail: dahlke@mathematik.uni-marburg.de\\
WWW: http://www.mathematik,uni-marburg.de/$\sim$dahlke/

\end{document}